\documentclass[11pt]{amsart}
\usepackage{amsmath,amssymb,amsthm,latexsym,cite,cancel}
\usepackage[small]{caption}
\usepackage{graphicx,wasysym,overpic,tikz,color}
\usepackage{subfigure,color}
\usepackage{cite}
\usepackage[colorlinks=true,urlcolor=blue,
citecolor=red,linkcolor=blue,linktocpage,pdfpagelabels,
bookmarksnumbered,bookmarksopen]{hyperref}
\usepackage[italian,english]{babel}
\usepackage{units}
\usepackage{enumitem}
\usepackage[left=2.1cm,right=2.1cm,top=2.71cm,bottom=2.71cm]{geometry}
\usepackage[hyperpageref]{backref}

\usepackage[colorinlistoftodos]{todonotes}
\newtheorem{theorem}{Theorem}[section]

\newtheorem{proposition}[theorem]{Proposition}
\newtheorem{lemma}[theorem]{Lemma}
\newtheorem{remark}[theorem]{Remark}

\newtheorem{corollary}[theorem]{Corollary}

\numberwithin{equation}{section}

\title[Magnetic nonlinear Schr\"odinger equation in $\mathbb{R}^2$]{Semiclassical states for a magnetic nonlinear Schr\"{o}dinger equation with exponential critical growth in $\mathbb{R}^{2}$}

\author[P. d'Avenia]{Pietro d'Avenia}

\address[P. d'Avenia]{\newline\indent
	Dipartimento di Meccanica, Matematica e Management
	\newline\indent
	Politecnico di Bari
	\newline\indent
	Via Orabona 4,  70125  Bari, Italy}
\email{\href{mailto:pietro.davenia@poliba.it}{pietro.davenia@poliba.it}}

\author[C. Ji]{Chao Ji}

\address[C. Ji]{\newline\indent
	Department of Mathematics
	\newline\indent
	East China University of Science and Technology
	\newline\indent
	Shanghai 200237, PR China }
\email{\href{mailto:jichao@ecust.edu.cn}{jichao@ecust.edu.cn}}

\date{\today}

\begin{document}

\begin{abstract}
This paper is devoted to the magnetic nonlinear Schr\"{o}dinger equation
\[
\Big(\frac{\varepsilon}{i}\nabla-A(x)\Big)^{2}u+V(x)u=f(| u|^{2})u \text{ in } \mathbb{R}^{2},
\]
where  $\varepsilon>0$ is a parameter,  $V:\mathbb{R}^{2}\rightarrow \mathbb{R}$ and $A: \mathbb{R}^{2}\rightarrow \mathbb{R}^{2}$ are continuous
functions and $f:\mathbb{R}\rightarrow \mathbb{R}$ is a $C^{1}$ function having exponential critical growth. Under a global assumption on the potential $V$, we use variational methods and Ljusternick-Schnirelmann theory to prove existence, multiplicity, concentration, and decay of nontrivial solutions for $\varepsilon>0$ small.
\end{abstract}

\maketitle

\begin{center}
	\begin{minipage}{11cm}
		\tableofcontents
	\end{minipage}
\end{center}

\section{Introduction and main results}

In this paper, we consider the following nonlinear Schr\"{o}dinger equation
\begin{equation}
\label{1.1}
\Big(\frac{\varepsilon}{i}\nabla-A(x)\Big)^{2}u+V(x)u=f(|u|^{2})u
\quad
\hbox{in }\mathbb{R}^2,
\end{equation}
where $\varepsilon>0$ is a parameter, $u:\mathbb{R}^2\to\mathbb{C}$, $A\in C(\mathbb{R}^{2}, \mathbb{R}^{2})$,
 $V\in C(\mathbb{R}^{2}, \mathbb{R})$ satisfies
\begin{equation}
\label{V}\tag{V}
V_{\infty}=\liminf_{\vert x\vert\rightarrow +\infty} V(x)>V_{0}=\inf_{x\in \mathbb{R}^{2}} V(x)>0,
\end{equation}
where $V_{\infty}$ can be also $+\infty$, and $f\in C^{1}(\mathbb{R}, \mathbb{R})$ satisfies the following assumptions:
\begin{enumerate}[label=($f_\arabic{*}$), ref=$f_\arabic{*}$]
	\item \label{f1}$f(t)=0$ if $t\leq0$;
	\item \label{f2}there holds
	\[
	\lim_{t\rightarrow +\infty} \frac{f(t^{2})t}{e^{\alpha t^{2}}}
	=
	\begin{cases}
		0,& \hbox{for } \alpha> 4\pi,\\
		+\infty,& \hbox{for }  0<\alpha<4\pi,
	\end{cases}
	\]
	({\em exponential critical growth});
	\item \label{f3}there is a positive constant $\theta>2$ such that, for every $t>0$,
	\begin{equation*}
		0<\frac{\theta}{2} F(t)\leq tf(t),
	\end{equation*}
where $F(t)=\displaystyle \int_{0}^{t}f(s)ds$;
	\item \label{f4}there exist two constants $p>2$ and
	\begin{equation*}
		C_{p}
		>
		\Big(\frac{8\theta}{\theta-2}\Big)^\frac{p-2}{2}
		\Big(\frac{p-2}{p}\Big)^\frac{p-2}{2}
		V^* S_{p}^{p/2}>0
	\end{equation*}
	such that
	\begin{equation*}
		f'(t)\geq \frac{p-2}{2}C_{p}t^{(p-4)/2}\quad \text{for all}\,\, t>0,
	\end{equation*}
	where
	\begin{equation}\label{SpV*}
		S_{p}=\inf_{u\in H^{1}(\mathbb{R}^{2}, \mathbb{R})\backslash\{0\}}\frac{\Vert \nabla u\Vert_2^{2}+\Vert u\Vert_2^{2}}{\Vert u\Vert_{p}^{2}}
		\text{ and }
		V^{*}=\left\{
		\begin{array}{l}
			V_{\infty},\quad \quad  \text{if}\,\, V_{\infty}<+\infty,\\
			V_{0},\quad \quad \,\, \text{if}\,\, V_{\infty}=+\infty;
		\end{array}%
		\right.
	\end{equation}
	\item \label{f5} $f'(t)\leq (e^{4\pi t}-1)$ for any $t\geq0$.
\end{enumerate}

Equation \eqref{1.1} appears in the study of standing wave solutions $\psi(x, t):=e^{-i\omega/\hbar}u(x)$, with $\omega\in \mathbb{R}$, of the planar nonlinear Schr\"odinger equation
\begin{equation*}
	i\hbar \frac{\partial \psi}{\partial t}=\Big(\frac{\hbar}{i}\nabla-A(x)\Big)^{2}\psi+U(x)\psi-f(|\psi|^{2})\psi
	\quad
	\hbox{in }\mathbb{R}^2\times\mathbb{R}.
\end{equation*}
In particular, the operator
\[
\psi \mapsto\Big( i\hbar \frac{\partial }{\partial t}-U(x)\Big)\psi -\Big(\frac{\hbar}{i}\nabla-A(x)\Big)^{2}\psi
\]
appears in the study of the interaction of a matter field $\psi$ with an external electromagnetic field whose electric potential is $U$ and whose magnetic potential is $A$ (here we are considering an electromagnetic potential which does not depend on the time variable) though the minimal coupling rule. Moreover, from a physical point of view, the study of the existence of solutions of \eqref{1.1}, in particular of {\color{blue} its } ground states, and of their behavior as $\varepsilon \to 0^+$ (which is equivalent to send the Planck constant $\hbar$ to zero and {\color{blue}it } is known as {\em semiclassical limit}) is of the greatest importance in the theory of superconductivity (see \cite{rDFS,rHM} and references therein) and since the transition from Quantum Mechanics to Classical Mechanics can be formally performed by this limit.

If the magnetic potential $A\equiv 0$ and $\mathbb{R}^{2}$ is replaced by $\mathbb{R}^{N}$, then \eqref{1.1} reduces to a following nonlinear Schr\"{o}dinger equation
\begin{equation}\label{1.11}
	-\varepsilon^2\Delta u+V(x)u=f(|u|^2)u
	\quad
	\hbox{in }\mathbb{R}^N.
\end{equation}
A vast literature is dedicated to \eqref{1.11}, in particular concerning existence and multiplicity of solutions, by several methods. If the potential $V$ satisfies the global assumption (\ref{V}), introduced by Rabinowitz in \cite{rRa}, Cingolani and Lazzo in \cite{rCL} considered multiplicity and concentration of positive solutions of \eqref{1.11} with $f(s)=s^{(p-2)/2}$ with $p>2$ if $N=1, 2$ and $2<p<2N/(N-2)$ if $N\geq 3$.
Then, Alves and Figueiredo in \cite{rAF2} generalized the results in \cite{rCL} to the $p$-Laplacian operator for a class of nonlinearities including power growth and in \cite{rAF3} considered multiplicity and concentration of positive solutions for $N$-Laplacian equation with  exponential critical growth in $\mathbb{R}^{N}$ for $N\geq 2$. Finally, let us mention \cite{rACTY, rAS}, that are devoted to critical cases in $\mathbb{R}^2$.

Moreover, after the seminal paper \cite{rEL}, also the  magnetic nonlinear Schr\"{o}dinger equation \eqref{1.1} has been extensively investigated by variational and topological methods (see \cite{ rAZ, rCi, rCJS, rCS1, rDS, rKu, rLS} and references therein).
In particular, in \cite{rAFF},  Alves, Figueiredo, and Furtado, using the penalization method and Ljusternik-Schnirelmann category theory, obtained a multiplicity result for \eqref{1.1} with a subcritical right hand side; in \cite{rBF}, Barile and Figueiredo considered the existence of a solution of \eqref{1.1} with an exponentially critically growing nonlinearity in $\mathbb{R}^{2}$; in \cite{rDJ}, we got multiplicity and concentration for \eqref{1.1} when the potential $V$ satisfies a local assumption due to del Pino and Felmer \cite{rDF}. Let us mention also the recent contributions \cite{rAMV} concerning a multiplicity result for a nonlinear fractional magnetic Schr\"{o}dinger equation with exponential critical growth in the one-dimensional case and \cite{rAD}, where existence and multiplicity for a nonlinear fractional magnetic Schr\"{o}dinger equation  in $\mathbb{R}^N$, $N\geq 3$, in the presence of a subcritical nonlinearity and of a potential $V$ satisfying a global assumption, has been proved.

The main result of this paper is the following.
\begin{theorem}\label{gt62}
Assume that \eqref{V} and (\ref{f1})--(\ref{f5}) hold. Then, there exists  $\varepsilon_{0}>0$ such that problem \eqref{1.1} admits a ground state
solution for any $\varepsilon\in (0, \varepsilon_{0})$.
\end{theorem}

Moreover, introducing the sets
\begin{equation}\label{M}
M:=\{x\in \mathbb{R}^{2}: V(x)=V_{0}\}
\end{equation}
and
\begin{equation*}
M_{\delta}:=\{x\in \mathbb{R}^{2}: \text{dist}(x, M)\leq \delta\}, \ \delta>0,
\end{equation*}
we are able to show, for small $\varepsilon>0$, a multiplicity result for problem \eqref{1.1}  and concentration phenomena of such solutions.
More precisely, we have
\begin{theorem}\label{gt666}
Assume that \eqref{V} and (\ref{f1})--(\ref{f5}) hold. Then, for any  $\delta>0$, there exists $\varepsilon_{\delta}>0$ such that problem \eqref{1.1}
has at least  $\text{cat}_{M_{\delta}}(M)$ nontrivial solutions, for any $0<\varepsilon<\varepsilon_{\delta}$.
\end{theorem}

Finally, for our solutions we are able to prove their concentration and decay. Indeed we have
\begin{theorem}\label{thm3}
Let $\{\varepsilon_{n}\}$ be a sequence in $(0,+\infty)$ such that $\varepsilon_{n}\rightarrow 0^{+}$, and, for $n$ large enough, let $\{u_{\varepsilon_{n}}\}$ be a sequence of solutions found in Theorem \ref{gt62} or in Theorem \ref{gt666}. If $\eta_{n}\in \mathbb{R}^{2}$ is a global maximum point of $\vert u_{\varepsilon_{n}}\vert$, then
\begin{equation*}
	\underset{n\rightarrow \infty}{\lim}V(\eta_{n})=V_{0}.
\end{equation*}
Moreover, if, furtherly, $A\in C^{1}(\mathbb{R}^{2}, \mathbb{R}^{2})$, then there exist $C, c>0$ such that
\begin{equation}
		\label{decay}
		\vert u_{\varepsilon_{n}}(x)\vert\leq C e^{-c \frac{|x-\eta_{n}|}{\varepsilon_{n}}}
		\text{ in } \mathbb{R}^{2}.
\end{equation}
\end{theorem}

We get our results combining variational methods and the Ljusternik-Schnirelmann theory.

This paper is a completion of our joint research project started with  \cite{rDJ}, where we assumed a local condition on $V$. Here the presence of the global condition \eqref{V} and the weaker assumption $A\in C(\mathbb{R}^{2}, \mathbb{R}^{2})$, require more delicate estimates and a careful and deeper analysis of some technical aspects. Moreover, comparing our case with known results obtained for $A\equiv 0$, the presence of the magnetic field leads to consider a complex valued  problem and this implies to get more delicate estimates. We are also able to weaken the assumptions of some technical lemmas (see, for instance, Lemma \ref{2.0}) and to study the behavior of the maximum points  of the modulus of solutions and their decay  both for the ground states obtained in Theorem \ref{gt62} and for solutions found in Theorem \ref{gt666}.

The paper is organized as follows. In Section \ref{Sec2} we introduce the functional setting and we give some preliminaries. We study the limit problem in Section \ref{Secthree}. Then, in Section \ref{Sec3}, we prove Theorem \ref{gt62}. Section \ref{Sec4} is devoted to prove Theorem \ref{gt666}. Finally, in the last section, we prove Theorem \ref{thm3}.

\subsection*{Notation}
\begin{itemize}
	\item $C, C_1, C_2, \ldots$ denote positive constants whose exact values can change from line to line;
	\item $B_{R}(y)$ denotes the open disk centered at $y\in\mathbb{R}^2$ with radius $R>0$ and $B^{c}_{R}(y)$ denotes its complement in $\mathbb{R}^{2}$;
	\item $\Vert \cdot \Vert$, $\Vert \cdot \Vert_{q}$, and $\Vert\cdot \Vert_{\infty}$ denote the usual norms of the spaces $H^{1}(\mathbb{R}^{2}, \mathbb{R})$, $L^{q}(\mathbb{R}^{2}, \mathbb{K})$, and $L^{\infty}(\mathbb{R}^{2}, \mathbb{K})$, respectively, where $\mathbb{K}=\mathbb{R}$ or $\mathbb{K}=\mathbb{C}$.
\end{itemize}

\section{The variational framework and some preliminaries}\label{Sec2}

In this section we introduce the functional spaces that we use and a {\em classical} equivalent version of \eqref{1.1}. We also present some results for a {\em limit problem} which will be useful for our arguments.

From now on, in the whole paper, we will assume that (\ref{V}) holds.

For a function $u: \mathbb{R}^{2}\rightarrow \mathbb{C}$, let us denote by
\begin{equation*}
\nabla_{A}u:=\Big(\frac{\nabla}{i}-A\Big)u
\end{equation*}
and consider the space
\begin{equation*}
H_{A}^{1}(\mathbb{R}^{2}, \mathbb{C}):=\{u\in L^{2}(\mathbb{R}^{2}, \mathbb{C}): |\nabla_{A}u|\in L^{2}(\mathbb{R}^{2}, \mathbb{R})\},
\end{equation*}
which is an Hilbert space endowed with the scalar product
\begin{equation*}
\langle u, v\rangle
:=\operatorname{Re}\int_{\mathbb{R}^{2}}\Big(\nabla_{A}u\overline{\nabla_{A}v}+u\overline{v}\Big)dx,\quad\text{for any } u, v\in H_{A}^{1}(\mathbb{R}^{2}, \mathbb{C}),
\end{equation*}
where $\operatorname{Re}$ and the bar denote the real part of a complex number and the complex conjugation, respectively. Hence, let $\Vert \cdot \Vert_{A}$ be the norm induced by this inner product.

Observe that a useful tool when we work on $H_{A}^{1}(\mathbb{R}^{2}, \mathbb{C})$ is the diamagnetic inequality
\begin{equation}
\label{2.1}
\vert \nabla_{A}u(x)\vert\geq \vert \nabla\vert u(x)\vert\vert
\end{equation}
(see e.g. \cite[Theorem~7.21]{rLL}).\\
Moreover, a simple change of variables allows us to write equation \eqref{1.1} as
\begin{equation}
\label{1.6}
\Big(\frac{1}{i}\nabla-A_{\varepsilon}(x)\Big)^{2}u+V_{\varepsilon}(x)u=f(\vert u\vert^{2})u
\quad
\hbox{in }\mathbb{R}^2,
\end{equation}
where $A_{\varepsilon}(x)=A(\varepsilon x)$ and $V_{\varepsilon}(x)=V(\varepsilon x)$.\\
Thus, we introduce the Hilbert space
 $$
 H_{\varepsilon}:=\Big\{u\in H_{A_{\varepsilon}}^{1}(\mathbb{R}^{2}, \mathbb{C}): \int_{\mathbb{R}^{2}}V_{\varepsilon}(x)\vert u\vert^{2}dx< \infty \Big\}
 $$
 endowed with the scalar product
 \begin{align*}
\langle u, v\rangle_{\varepsilon}
:=\operatorname{Re}\int_{\mathbb{R}^{2}}\Big(\nabla_{A_{\varepsilon}}u\overline{\nabla_{A_{\varepsilon}}v}+V_{\varepsilon}(x)u\overline{v}\Big)dx
\end{align*}
and $\Vert \cdot \Vert_{\varepsilon}$ is the norm induced by $\langle \cdot , \cdot \rangle_{\varepsilon}$.

For every $\varepsilon>0$, the functional space $H_{\varepsilon}$ satisfies the following embeddings. For the sake of completeness here we give some details about their proof.

\begin{lemma}\label{lem21}
The space $H_{\varepsilon}$ is continuously embedded in $L^{r}(\mathbb{R}^{2}, \mathbb{C})$ for $r\geq 2$ and compact embedded in $L^{r}_{\rm{loc}}(\mathbb{R}^{2}, \mathbb{C})$ for $r\geq 1$.
Moreover, if $V_{\infty}=+\infty$, then $H_{\varepsilon}$ is compactly embedded in $L^{r}(\mathbb{R}^{2}, \mathbb{C})$ for $r\geq 2$.
\end{lemma}
\begin{proof}
According to the diamagnetic inequality \eqref{2.1}, if $u\in H_{\varepsilon}$, then $|u|\in H^{1}(\mathbb{R}^{2}, \mathbb{R})$  and
\begin{equation}
\label{emb}
\| |u| \|\leq C \|u\|_\varepsilon.
\end{equation}
Thus, the continuous embedding $H_{\varepsilon}\hookrightarrow L^{r}(\mathbb{R}^{2}, \mathbb{C})$ is a consequence of \eqref{emb} and of the continuous embedding  $H^{1}(\mathbb{R}^2,\mathbb{R})\hookrightarrow L^{r}(\mathbb{R}^{2}, \mathbb{R})$ for $r\geq 2$.\\
To get the compact embedding of $H_{\varepsilon}$ in $L^{r}_{\rm{loc}}(\mathbb{R}^{2}, \mathbb{C})$ for $r\geq 1$ we can argue in the following way.\\
Let $K$ be an open subset of $\mathbb{R}^2$ with compact closure.\\
First of all, we prove that $H_\varepsilon$ is continuously embedded in $H^1(K,\mathbb{C})$.\\
Indeed we have
\begin{align*}
\|u\|_{H^1(K,\mathbb{C})}^2
&=
\int_K |\nabla u|^2 dx
+ \int_K |u|^2 dx
= \int_K \Big| \frac{\nabla}{i} u \Big|^2 dx
+ \int_K |u|^2 dx \\
&\leq
C\left(\int_K | \nabla_{A_\varepsilon} u |^2 dx
+ \int_K | A_\varepsilon u |^2 dx
+ \int_K V_\varepsilon (x) |u|^2 dx \right)\\
&\leq
C\left(
\int_K | \nabla_{A_\varepsilon} u |^2 dx
+ \int_K V_\varepsilon (x) |u|^2 dx \right)
\leq C \| u\|_\varepsilon^2.
\end{align*}
Let now $\mathcal{F}$ be a bounded subset of $H_\varepsilon$. We want to prove that $\mathcal{F}|_K$, which is a subset of $H^1(K,\mathbb{C})$ due to the continuous embedding proved before, has compact closure in $L^r(K,\mathbb{C})$, $r\geq 1$.\\
Let $\tilde{\mathcal{P}}$ be an extension operator found in \cite[Theorem 9.7]{rBr} for $H^1(K,\mathbb{R})$ and $\mathcal{P}:= \tilde{\mathcal{P}}({\rm Re}(u))+i\tilde{\mathcal{P}}({\rm Im}(u))$.\\
We know that, for each $v\in H^1(K,\mathbb{R})$,
$$\tilde{\mathcal{P}}v|_{K}=v, \
\|\tilde{\mathcal{P}}v\|_{L^2(\mathbb{R}^2,\mathbb{R})}\leq C\|v\|_{L^2(K,\mathbb{R})}
\text{ and }
\|\tilde{\mathcal{P}}v\|_{H^1(\mathbb{R}^2,\mathbb{R})}\leq C\|v\|_{H^1(K,\mathbb{R})}.$$
Thus, we get
\[
\mathcal{P}u|_{K}
= \tilde{\mathcal{P}}({\rm Re}(u))|_{K}+i\tilde{\mathcal{P}}({\rm Im}(u))|_{K}
= {\rm Re}(u) +i {\rm Im}(u)
= u,
\]
\begin{align*}
\|\mathcal{P}u\|_{L^2(\mathbb{R}^2,\mathbb{C})}^2
&=
\|\tilde{\mathcal{P}}({\rm Re}(u))+i\tilde{\mathcal{P}}({\rm Im}(u))\|_{L^2(\mathbb{R}^2,\mathbb{C})}^2
\leq
C(
\|{\rm Re}(u)\|_{L^2(K,\mathbb{R})}^2
+ \|{\rm Im}(u)\|_{L^2(K,\mathbb{R})}^2
)
\\
&=
C\|u\|_{L^2(K,\mathbb{C})}^2
\end{align*}
and, analogously,
\begin{align*}
\|\mathcal{P}u\|_{H^1(\mathbb{R}^2,\mathbb{C})}^2
&=
\|\tilde{\mathcal{P}}({\rm Re}(u))+i\tilde{\mathcal{P}}({\rm Im}(u))\|_{H^1(\mathbb{R}^2,\mathbb{C})}^2
\leq
C(
\|{\rm Re}(u)\|_{H^1(K,\mathbb{R})}^2
+ \|{\rm Im}(u)\|_{H^1(K,\mathbb{R})}^2
)
\\
&=
C\|u\|_{H^1(K,\mathbb{C})}^2.
\end{align*}
Let now $\mathcal{H}:=\mathcal{P}(\mathcal{F}|_K)$.\\
Since $H_\varepsilon$ is continuously embedded in $H^1(K,\mathbb{C})$, we have that $\mathcal{H}$ is bounded in $H^1(\mathbb{R}^2,\mathbb{C})$ and $\mathcal{H}|_K=\mathcal{F}|_K$.\\
Thus we can proceed in a classical way, see e.g. \cite[Theorem 9.16]{rBr}, using \cite[Theorem 4.26]{rBr}.\\
Indeed, first observe that, since $K$ is bounded, we may assume $r\geq 2$.\\
Then, let $u\in C_{c}^{\infty}(\mathbb{R}^{2}, \mathbb{C})$ and $\tau_h u:=u(\cdot+h)$, with $h\in\mathbb{R}^2$.\\
For every $x\in\mathbb{R}^2$ we have
\begin{align*}
|\tau_h u (x) - u(x)|
&=
\sqrt{[{\rm Re}(u)(x+h)-{\rm Re}(u)(x)]^2+[{\rm Im}(u)(x+h)-{\rm Im}(u)(x))]^2}\\
&\leq
|{\rm Re}(u)(x+h)-{\rm Re}(u)(x)|
+|{\rm Im}(u)(x+h)-{\rm Im}(u)(x)|\\
&\leq
\int_0^1 |v_{\rm R}'(t)| dt
+ \int_0^1 |v_{\rm I}'(t)| dt\\
&=
\int_0^1 |h\cdot\nabla {\rm Re}(u)(x+th)| dt
+ \int_0^1 |h\cdot\nabla {\rm Im}(u)(x+th)| dt\\
&\leq
|h| \left(\int_0^1 (|\nabla {\rm Re}(u)(x+th)|
+ |\nabla {\rm Im}(u)(x+th)|) dt\right)\\
&\leq
\sqrt{2}|h| \int_0^1 |\nabla u(x+th)| dt
\end{align*}
where $v_{\rm R}(t):={\rm Re}(u)(x+th)$ and $v_{\rm I}(t):={\rm Im}(u)(x+th)$, and so, using the H\"older inequality,
\[
|\tau_h u (x) - u(x)|^2 \leq 2|h|^2 \int_0^1 |\nabla u(x+th)|^2 dt.
\]
Hence, using the density of  $C_{c}^{\infty}(\mathbb{R}^{2}, \mathbb{C})$ in $H^1(\mathbb{R}^2,\mathbb{C})$, see \cite[Proposition 2.1]{rEL},
\[
\|\tau_h u - u\|_2 \leq C|h| \|\nabla u\|_2 \leq C |h|
\]
uniformly with respect to $u\in\mathcal{H}$.\\
For $r>2$, if we take $q>r$, there exists $\alpha\in (0,1)$ such that
\[
\frac{1}{r} = \frac{\alpha}{2} + \frac{1-\alpha}{q},
\]
so that, for every $u\in \mathcal{H}$,
\[
\|\tau_h u - u\|_r
\leq
\|\tau_h u - u\|_2^\alpha \|\tau_h u - u\|_q^{1-\alpha}
\leq C |h|^\alpha
\]
with $C$ independent of the elements of $\mathcal{H}$.\\
Thus, \cite[Theorem 4.26]{rBr} implies that the closure of $\mathcal{H}|_K=\mathcal{F}|_K$ is compact in $L^r(K,\mathbb{C})$.\\
Finally, if $V_{\infty}=+\infty$, let $\{u_n\}\subset H_\varepsilon$ be a sequence weakly convergent to $u$ in  $H_\varepsilon$ and $v_n:=u_n -u$.\\
We have that $v_n \to 0$ in $L^2(\mathbb{R}^2,\mathbb{C})$.\\
Indeed, for every $\eta>0$, there exists $\sigma>0$  large enough such that
\[
\frac{1}{V(x)} <\eta
\text{ in } B_\sigma^c(0).
\]
Moreover, since $v_n \rightharpoonup 0$ in  $H_\varepsilon$, then $v_n \to 0$ in $L_{\rm loc}^2(\mathbb{R}^2,\mathbb{C})$.
Thus, for $n$ large enough,
\[
\| v_n\|_2^2
< \eta
+ \frac{\eta}{V_{0}} \int_{B_{\sigma/\varepsilon}^c(0)} V_\varepsilon(x) |v_n|^2 dx
< C\eta.
\]
Hence, using the Gagliardo-Nirenberg inequality applied to $|v_n|$, for every $r \geq 2$ we have
\[
\| v_n\|_r \leq C \|\nabla |v_n|\|_2^{1-2/r} \| v_n\|_2^{2/r},
\]
and, by the diamagnetic inequality \eqref{2.1} and the boundedness of $\{v_n\}$ in $H_\varepsilon$, we conclude.
\end{proof}

For compact supported functions in $H^{1}(\mathbb{R}^{2}, \mathbb{R})$, we have the following result which will be very important for some estimates below.
\begin{lemma}\label{gt520}
If $u\in H^{1}(\mathbb{R}^{2}, \mathbb{R})$ and $u$ has a compact support, then $\omega:=e^{iA(0)\cdot x}u\in H_{\varepsilon}$.
\end{lemma}
\begin{proof}
Assume that $\text{supp} (u)\subset B_{R}(0)$. Since $V$ is continuous, it is clear that
\begin{align*}
\int_{\mathbb{R}^{2}}V_{\varepsilon}(x)\vert \omega\vert^{2}dx=\int_{B_{R}(0)}V_{\varepsilon}(x)\vert \omega\vert^{2}dx\leq C\Vert u\Vert_{2}^{2}<+\infty.
\end{align*}
Moreover, since $V$ and $A$ are continuous, we have
\begin{align*}
\int_{\mathbb{R}^{2}}\vert \nabla_{A_{\varepsilon}} \omega\vert^{2} dx
&=
\int_{\mathbb{R}^{2}}\vert\nabla\omega\vert^{2} dx
+\int_{\mathbb{R}^{2}}\vert A_{\varepsilon}(x)\vert^{2}\vert
\omega\vert^{2}dx
+2\text{Re}\int_{\mathbb{R}^{2}}iA_{\varepsilon}(x)\overline{\omega}\nabla \omega dx\\
&\leq
2\int_{\mathbb{R}^{2}}\vert\nabla\omega\vert^{2}dx
+2 \int_{\mathbb{R}^{2}} \vert A_{\varepsilon}(x)\vert^{2}\vert \omega\vert^{2} dx\\
&\leq
C\left[
\int_{\mathbb{R}^{2}} \vert \nabla u\vert^{2} dx
+ \int_{\mathbb{R}^{2}} \vert u\vert^{2} dx
\right] <+\infty
\end{align*}
and we conclude.
\end{proof}

A useful tool when we work with nonlinearities with exponential critical growth is the following version of Trudinger-Moser inequality as stated e.g. in \cite[Lemma 1.2]{ rACTY} (see also \cite[Lemma 2.1]{rC}).
\begin{lemma}\label{le24}
	If $\alpha>0$ and $u\in H^{1}(\mathbb{R}^{2}, \mathbb{R})$, then
	\begin{equation*}
	\int_{\mathbb{R}^{2}}(e^{\alpha \vert u\vert^{2}}-1)dx<+\infty.
	\end{equation*}
	Moreover, if $\Vert \nabla u\Vert_{2}^{2}\leq 1$, $\Vert  u\Vert_{2}\leq M<+\infty$, and $0<\alpha< 4\pi$, then there exists a positive constant $C(M, \alpha)$, which depends only on $M$ and $\alpha$,  such that
	\begin{equation*}
	\int_{\mathbb{R}^{2}}(e^{\alpha \vert u\vert^{2}}-1)dx\leq C(M, \alpha).
	\end{equation*}
\end{lemma}

Observe also that that (\ref{f1}) and (\ref{f2}) imply that, fixed $q>2$, for any $\zeta>0$ and $\alpha>4\pi$, there exists a constant $C>0$, which depends on $q$, $\alpha$, $\zeta$, such that
	\begin{equation}
	\label{2.8}
	f(t)\leq\zeta+Ct^{(q-2)/2}(e^{\alpha t}-1),\, \text{ for all } t\geq 0
	\end{equation}
	and, using (\ref{f3}),
	\begin{equation}
	\label{2.9}
	F(t)\leq\zeta t+Ct^{q/2}(e^{\alpha t}-1),\, \text{ for all } t\geq 0.
	\end{equation}
	Then, by \eqref{2.8} and \eqref{2.9}, we have
	\begin{equation}
	\label{2.10}
	f(t^{2})t^{2}\leq\zeta t^{2}+C\vert t\vert^{q}(e^{\alpha  t^{2}}-1),\, \text{ for all } t\in\mathbb{R}
	\end{equation}
	and
	\begin{equation}
	\label{2.11}
	F(t^{2})\leq\zeta t^{2}+C\vert t\vert^{q}(e^{\alpha t^{2}}-1),\,
	\text{ for all } t\in\mathbb{R}.
	\end{equation}
Moreover (\ref{f4}) and (\ref{f5}) imply that there exist two positive constants $C_1$ and $C_2$ such that
	$$
	C_1 t^{(p-4)/2}\leq f'(t)\leq C_{2} t,\quad \text{as } t\rightarrow 0^{+}
	$$
	and so $p> 6$.

For $u\in H_{\varepsilon}$, let
\begin{equation}
\label{2.2}
\hat{u}_{j}:=\varphi_{j}u,
\end{equation}
where $j\in \mathbb{N}^{*}$ and $\varphi_{j}(x)=\varphi(2x/j)$ with $\varphi\in C_{c}^{\infty}(\mathbb{R}^{2}, \mathbb{R})$, $0\leq \varphi\leq 1$,
$\varphi(x)=1$ if $\vert x\vert\leq 1$, and $\varphi(x)=0$ if $\vert x\vert\geq 2$.\\
Note that $\hat{u}_{j}$ belongs to $H_{\varepsilon}$ and has compact support.\\
We have  that any $u\in H_{\varepsilon}$ can be {\em approximated} by $\hat{u}_j$ in the following sense. This property is classical, but, for completeness, we give here some details.
\begin{lemma}\label{gt521}
	Let $\varepsilon>0$. For any $u\in H_{\varepsilon}$, $\Vert \hat{u}_{j}-u\Vert_{\varepsilon}\rightarrow 0$ as $j\rightarrow +\infty$.
\end{lemma}
\begin{proof}
	Since $u\in H_{\varepsilon}$, by the Lebesgue Dominated Convergence Theorem, it is easy to see that
	\[
	\int_{\mathbb{R}^{2}}V_{\varepsilon}(x)\vert \hat{u}_{j}-u\vert^{2}dx=o_{j}(1).
	\]
	Moreover, since
	\begin{equation*}
	\begin{split}
	\int_{\mathbb{R}^{2}}\vert \nabla_{A_{\varepsilon}} (\hat{u}_{j}-u)\vert^{2} dx
	&=
	\int_{\mathbb{R}^{2}}\vert \nabla (\hat{u}_{j}-u)\vert^{2} dx
	+ \int_{\mathbb{R}^{2}} \vert A_{\varepsilon}(x)\vert^{2}\vert \hat{u}_{j}-u\vert^{2} dx\\
	&\qquad
	+2\text{Re}\int_{\mathbb{R}^{2}}i \nabla (\hat{u}_{j}-u) A_{\varepsilon}(x)\overline{(\hat{u}_{j}-u)}dx\\
	&\leq
	2\int_{\mathbb{R}^{2}}\vert \nabla (\hat{u}_{j}-u)\vert^{2} dx
	+2\int_{\mathbb{R}^{2}} \vert A_{\varepsilon}(x)\vert^{2}\vert \hat{u}_{j}-u\vert^{2}dx,
	\end{split}
	\end{equation*}
	applying again the Lebesgue Dominated Convergence Theorem, we conclude.
\end{proof}
In the spirit of \cite[Lemma 3.2]{rDL}, we have the following result for bounded sequences in $H_{\varepsilon}$.
\begin{lemma}\label{2.01}
Let $s\geq 2$ and  $\{u_{n}\}\subset H_{\varepsilon}$ be a bounded sequence. Then there exists a subsequence $\{u_{n_{j}}\}$ such that for any   $\sigma>0$ there is $r_{\sigma, s}>0$ such that, for every $r\geq r_{\sigma, s}$,
\[
\limsup_{j} \int_{B_{j}(0)\backslash B_{r}(0)}\vert u_{n_{j}}\vert^{s}dx\leq \sigma.
\]
\end{lemma}
\begin{proof}
By Lemma \ref{lem21}, for each $j\in \mathbb{N}$,
$$\int_{B_{j}(0)}\vert u_{n}\vert^{s}dx\rightarrow \int_{B_{j}(0)}\vert u\vert^{s}dx \text{ as } n\rightarrow+\infty.$$
Thus there exists $\nu_{j}\in \mathbb{N}$ such that
$$
\Big|\int_{B_{j}(0)}(\vert u_{n}\vert^{s}-\vert u\vert^{s})dx\Big|< \frac{1}{j}
\quad
\text{ for all } n>\nu_{j}
$$
and, without loss of generality, we can assume that $\nu_{j+1}\geq \nu_{j}$.\\
In particular, for $n_{j}:=\nu_{j}+j$ we have
$$
\int_{B_{j}(0)}(\vert u_{n_{j}}\vert^{s}-\vert u\vert^{s})dx< \frac{1}{j}.
$$
On the other hand, there exists $r_{\sigma, s}$ such that, for every $r\geq r_{\sigma, s}$,
$$
\int_{\mathbb{R}^{2}\backslash B_{r}(0)}\vert u\vert^{s}\leq \frac{\sigma}{3}.
$$
Hence, for $j$ large enough,
\begin{align*}
 \int_{B_{j}(0)\backslash B_{r}(0)}\vert u_{n_{j}}\vert^{s}dx
&= \int_{B_{j}(0)}(\vert u_{n_{j}}\vert^{s}-\vert u\vert^{s})dx
+\int_{B_{j}(0)\backslash B_{r}(0)}\vert u\vert^{s}dx
+\int_{B_{r}(0)}(\vert u\vert^{s}-\vert u_{n_{j}}\vert^{s})dx\\
& \leq \frac{1}{j}
+\frac{\sigma}{3}
+\frac{1}{j}
\leq \sigma.
\end{align*}
\end{proof}

In the proofs of our results, the following properties will be useful.
\begin{lemma}\label{2.0}
Let $\{u_{n}\}\subset H_{\varepsilon}$ be a bounded sequence  and $\{u_{n_{j}}\}$ be a subsequence of $\{u_{n}\}$ such that $u_{n_{j}}\rightharpoonup u$ in $H_{\varepsilon}$ and $\nabla\vert u_{n_{j}}\vert\rightarrow \nabla\vert u\vert$ a.e. in $\mathbb{R}^{2}$, as $j\to +\infty$. Moreover let $\ell^2:=\limsup_j\Vert \nabla\vert u_{n_{j}}\vert \Vert_{2}^{2}$ and $\hat{u}_{j}$ be defined as in \eqref{2.2}.\\
Then, up to a subsequence still denoted by $\{u_{n_j}\}$, we have that, if $\ell^2<1/4$,
\begin{equation}
\label{apprF}
\int_{\mathbb{R}^{2}}\Big( F(\vert u_{n_{j}}\vert^{2})-F(\vert u_{n_{j}}-\hat{u}_{j}\vert^{2})-F(\vert \hat{u}_{j}\vert^{2})\Big) dx=o_{j}(1)
\end{equation}
and, for every $\phi\in H_\varepsilon$,
if $\ell^2<1/8$,
\begin{equation}
\label{apprf}
\int_{\mathbb{R}^{2}} \Big\vert f(\vert u_{n_{j}}\vert^{2})u_{n_{j}}-f(\vert u_{n_{j}}-\hat{u}_{j}\vert^{2})(u_{n_{j}}-\hat{u}_{j})- f(\vert
\hat{u}_{j}\vert^{2})\hat{u}_{j} \Big\vert^2 =o_{j}(1).
\end{equation}

\end{lemma}

\begin{proof}
Let us prove \eqref{apprF}.\\
Using \eqref{2.8} and the Young inequality, we have that, fixed $q>2$, for any $\zeta>0$ and $\alpha>4\pi$, there exists a constant $C>0$, which depends on $q$, $\alpha$, $\zeta$, such that
\begin{align*}
\vert F(\vert u_{n_{j}}\vert^{2})-F(\vert u_{n_{j}}-\hat{u}_{j}\vert^{2})\vert
&\leq 2\int_{0}^{1}\vert f(\vert u_{n_{j}}-t\hat{u}_{j}\vert^{2})\vert  \vert
u_{n_{j}}-t\hat{u}_{j}\vert\vert \hat{u}_{j}\vert dt\\
& \leq
\zeta \int_{0}^1 \vert u_{n_{j}}-t\hat{u}_{j}\vert \vert \hat{u}_{j}\vert dt
+C \int_{0}^1 \vert u_{n_{j}}-t\hat{u}_{j}\vert^{q-1}
(e^{\alpha \vert u_{n_{j}}-t\hat{u}_{j}\vert^2} -1) \vert \hat{u}_{j}\vert dt
\\
&\leq \zeta (\vert u_{n_{j}}\vert+\vert u \vert)\vert u \vert
+C (\vert u_{n_{j}}\vert+\vert u \vert)^{q-1}  (e^{\alpha (\vert u_{n_{j}}\vert+\vert u \vert)^{2}}-1) \vert u \vert \\
&\leq \zeta\big(\vert u_{n_{j}}\vert^{2}+\vert u_{n_{j}}\vert^{q}(e^{\alpha (\vert u_{n_{j}}\vert+\vert u\vert)^{2}}-1)\big)+C\big(\vert u\vert^{2}+\vert
u\vert^{q}(e^{\alpha (\vert u_{n_{j}}\vert+\vert u\vert)^{2}}-1)\big).
\end{align*}
Then, by \eqref{2.11} we have
\begin{equation}\label{Fnj}
\begin{split}
\vert F(\vert u_{n_{j}}\vert^{2})-F(\vert u_{n_{j}}-\hat{u}_{j}\vert^{2})-F(\vert \hat{u}_{j}\vert^{2})\vert
&\leq
\zeta\big(\vert u_{n_{j}}\vert^{2}+\vert
u_{n_{j}}\vert^{q}(e^{\alpha (\vert u_{n_{j}}\vert+\vert u\vert)^{2}}-1)\big)\\
&\quad+C\big(\vert u\vert^{2}+\vert u\vert^{q}(e^{\alpha (\vert u_{n_{j}}\vert+\vert u\vert)^{2}}-1)\big).
\end{split}
\end{equation}
Now, let
\begin{align*}
G_{j}^{\zeta}:=\max\Big\{\vert F(\vert u_{n_{j}}\vert^{2})-F(\vert u_{n_{j}}-\hat{u}_{j}\vert^{2})-F(\vert \hat{u}_{j}\vert^{2})\vert- \zeta(\vert
u_{n_{j}}\vert^{2}+\vert u_{n_{j}}\vert^{q}(e^{\alpha (\vert u_{n_{j}}\vert+\vert u\vert)^{2}}-1)), 0 \Big\}.
\end{align*}
Note that, by \eqref{Fnj}, we have
\begin{equation}
\label{bddGjz}
G_{j}^{\zeta}
\leq
C\big(\vert u\vert^{2}+\vert u\vert^{q}(e^{\alpha (\vert u_{n_{j}}\vert+\vert u\vert)^{2}}-1)\big)
\end{equation}
and, by Lemma \ref{gt521} and Lemma \ref{lem21},
\begin{equation}
\label{Gjzae}
G_{j}^{\zeta}\rightarrow 0
\text{ a.e. in }
\mathbb{R}^{2}
\end{equation}
and
\begin{equation}
\label{eunj}
e^{\alpha (\vert u_{n_{j}}\vert+\vert u\vert)^{2}}
\to
e^{4\alpha \vert u\vert^{2}}
\text{ a.e. in }
\mathbb{R}^{2}
\end{equation}
as $j\rightarrow +\infty$.\\
	Let us show that
	\begin{equation}
	\label{claim25}
	\vert u\vert^{q}(e^{\alpha (\vert u_{n_{j}}\vert+\vert u\vert)^{2}}-1)\rightarrow \vert u\vert^{q}(e^{4\alpha\vert u\vert^{2}}-1)\quad\text{in}\,\,
	L^{1}(\mathbb{R}^{2}, \mathbb{R}).
	\end{equation}
Since $\nabla \vert u_{n_{j}}\vert\rightarrow \nabla\vert u\vert$ a.e. in $\mathbb{R}^{2}$ as $j\rightarrow + \infty$, by the Fatou's Lemma, we deduce that
\begin{equation*}
	\Vert \nabla\vert u\vert\Vert_2^{2}
	\leq
	\liminf_j\Vert \nabla\vert u_{n_{j}}\vert\Vert_2^{2}\leq \ell^2
\end{equation*}
and so, for any fixed $\kappa\in(0,1-4\ell^2)$, we have that for $j$ large enough,
\begin{equation}\label{1/4}
	\| \nabla (\vert u_{n_{j}}\vert+ \vert u\vert) \|_2^{2}
	\leq
	2 	\| \nabla \vert u_{n_{j}}\vert \|_2^{2}
	+2	\|\nabla \vert u\vert \|_2^{2}
	<4\ell^2+\kappa.
\end{equation}
Let
\[
p^*\in\Big(1,\frac{1}{4\ell^2+\kappa}\Big),
\qquad
\alpha\in\Big(4\pi,\frac{4\pi}{p^*(4\ell^2+\kappa)}\Big).
\]
By the inequality
\begin{equation}\label{ineqe}
	(e^{t}-1)^{s}\leq e^{ts}-1, \text{ for }s>1 \text{ and } t\geq 0,
\end{equation}
inequality \eqref{1/4}, and Lemma  \ref{le24},
we get
\begin{equation}
	\label{bddinte}
	\int_{\mathbb{R}^{2}}(e^{\alpha (\vert u_{n_{j}}\vert+\vert u\vert)^{2}}-1)^{p^*}dx
	\leq
	\int_{\mathbb{R}^{2}}(e^{(4\ell^2+\kappa) p^* \alpha \Big(\frac{\vert u_{n_{j}}\vert+\vert u\vert}{\sqrt{4\ell^2+\kappa}}\Big)^{2}}-1)dx
	\leq C.
\end{equation}
Thus, up to a subsequence, and applying the Severini-Egoroff Theorem \cite[Theorem 4.29]{rBr},
	\[
	(e^{\alpha (\vert u_{n_{j}}\vert+\vert u\vert)^{2}}-1)
	\rightharpoonup
	(e^{4\alpha \vert u\vert^{2}}-1)
	\text{ in } L^{p^*}(\mathbb{R}^{2}, \mathbb{R})
	\]
	and, since $\vert u\vert^{q}\in L^{q^*}(\mathbb{R}^{2})$, where $q^*$ is the conjugate exponent of $p^*$, we get \eqref{claim25}.\\
Hence, by \eqref{bddGjz}, \eqref{Gjzae}, \eqref{eunj}, and \eqref{claim25}, applying a variant of the Lebesgue Dominated Convergence Theorem, we deduce that
\begin{align*}
\int_{\mathbb{R}^{2}}G_{j}^{\zeta}dx\rightarrow 0\quad \text{as}\,\,j\rightarrow +\infty.
\end{align*}
On the other hand, by the definition of $G_{j}^{\zeta}$,
\begin{align*}
\vert F(\vert u_{n_{j}}\vert^{2})
-F(\vert u_{n_{j}}-\hat{u}_{j}\vert^{2})
-F(\vert\hat{u}_{j}\vert^{2})\vert
\leq
\zeta (\vert u_{n_{j}}\vert^{2}+\vert u_{n_{j}}\vert^{q}(e^{\alpha (\vert u_{n_{j}}\vert+\vert u\vert)^{2}}-1))+G_{j}^{\zeta}.
\end{align*}
Using the H\"older inequality, the boundedness of the sequence $\{u_n\}$, Lemma \ref{lem21}, and \eqref{bddinte}, we have
\[
	\int_{\mathbb{R}^{2}}\Big(\vert u_{n_{j}}\vert^{q}(e^{\alpha (\vert u_{n_{j}}\vert+\vert u\vert)^{2}}-1)\Big) dx
	\leq
	\|u_{n_j}\|_{qq^*}^q
	\left[
	\int_{\mathbb{R}^{2}}\vert (e^{\alpha (\vert u_{n_{j}}\vert+\vert u\vert)^{2}}-1)^{p^*}dx
	\right]^{\frac{1}{p^*}}
	\leq C.
	\]
Thus,  the arbitrariness of $\zeta$ allows to conclude the proof of \eqref{apprF}.\\
Now we prove \eqref{apprf}.\\
First observe that, by Lemma \ref{lem21} and Lemma \ref{gt521}, we have that
\[
\Big\vert f(\vert u_{n_{j}}\vert^{2})u_{n_{j}}-f(\vert u_{n_{j}}-\hat{u}_{j}\vert^{2})(u_{n_{j}}-\hat{u}_{j})- f(\vert
\hat{u}_{j}\vert^{2})\hat{u}_{j} \Big\vert \to 0
\text{ a.e. in } \mathbb{R}^2.
\]
Moreover, using also \eqref{2.8}, we have that fixed $q>2$, for any $\zeta>0$ and $\alpha>4\pi$, there exists a constant $C>0$, which depends on $q$, $\alpha$, $\zeta$, such that
\begin{align*}
&\Big\vert f(\vert u_{n_{j}}\vert^{2})u_{n_{j}}-f(\vert u_{n_{j}}-\hat{u}_{j}\vert^{2})(u_{n_{j}}-\hat{u}_{j})- f(\vert
\hat{u}_{j}\vert^{2})\hat{u}_{j} \Big\vert\\
&\qquad\leq
\zeta |u_{n_{j}}|
+ C |u_{n_{j}}|^{q-1} (e^{\alpha|u_{n_{j}}|^2}-1)
+ \zeta |u_{n_{j}}-\hat{u}_{j}|
+ C |u_{n_{j}}-\hat{u}_{j}|^{q-1} (e^{\alpha|u_{n_{j}}-\hat{u}_{j}|^2}-1)\\
&\qquad\qquad
+ \zeta |\hat{u}_{j}|
+ C |\hat{u}_{j}|^{q-1} (e^{\alpha|\hat{u}_{j}|^2}-1)\\
&\qquad\leq
C\left[
|u_{n_{j}}| + |u|
+ (|u_{n_{j}}| + |u|)^{q-1} (e^{\alpha(|u_{n_{j}}|+|u|)^2}-1)
\right]
\end{align*}
and so
\begin{equation}
\label{fnj}
\begin{split}
&\Big\vert f(\vert u_{n_{j}}\vert^{2})u_{n_{j}}-f(\vert u_{n_{j}}-\hat{u}_{j}\vert^{2})(u_{n_{j}}-\hat{u}_{j})- f(\vert
\hat{u}_{j}\vert^{2})\hat{u}_{j} \Big\vert^2\\
&\qquad\leq
C\left[
|u_{n_{j}}|^2 + |u|^2
+ (|u_{n_{j}}| + |u|)^{2(q-1)} (e^{\alpha(|u_{n_{j}}|+|u|)^2}-1)^2
\right].
\end{split}
\end{equation}
Arguing as before, for any fixed $\kappa\in(0, \frac{1}{2}-4\ell^2)$, if
	\[
	p^*\in\Big(1,\frac{1}{2(4 \ell^2+\kappa)}\Big)
	\text{ and }
	\alpha\in\Big(4\pi,\frac{2\pi}{p^*(4\ell^2+\kappa)}\Big),
	\]
	we have that
	\begin{equation}
		\label{bddintep}
		\int_{\mathbb{R}^{2}}(e^{\alpha (\vert u_{n_{j}}\vert+\vert u\vert)^{2}}-1)^{2p^*}dx
		\leq \int_{\mathbb{R}^{2}}(e^{2(4\ell^2+\kappa) p^*  \alpha \Big(\frac{\vert u_{n_{j}}\vert+\vert u\vert}{\sqrt{4\ell^2+\kappa}}\Big)^{2}}-1)dx
		\leq C.
\end{equation}
By Lemma \ref{2.01} for $s=2$, from $\{u_{n_j}\}$ we can extract a first subsequence, still denoted by $\{u_{n_j}\}$, such that for any  $\sigma>0$ there is $r_{\sigma, 2}>0$ such that, for every $r\geq r_{\sigma, 2}$,
\begin{equation}
\label{Bjr2sigma}
\limsup_{j} \int_{B_{j}(0)\backslash B_{r}(0)}\vert u_{n_{j}}\vert^{2}dx\leq \sigma.
\end{equation}
Applying again Lemma \ref{2.01} for $s=2(q-1)q^*$, where $q^*$ is the conjugate exponent of $p^*$, from $\{u_{n_j}\}$ we can extract a further subsequence, still denoted by $\{u_{n_j}\}$, such that for any  $\sigma>0$ there is $r_{\sigma, 2(q-1)q^*}>0$ such that, for every $r\geq r_{\sigma, 2(q-1)q^*}$,
\begin{equation}
\label{Bjrqsigma}
\limsup_{j} \int_{B_{j}(0)\backslash B_{r}(0)}\vert u_{n_{j}}\vert^{2(q-1)q^*}dx\leq \sigma.
\end{equation}
Let now $\sigma>0$ be arbitrary and $r\geq \max\{r_{\sigma, 2}, r_{\sigma, 2(q-1)q^*}\}$ large enough such that
\begin{equation}
\label{Brcsigma}
\int_{B^{c}_{r}(0)}\vert u\vert^{2}dx\leq \sigma
\text{ and }
\int_{B^{c}_{r}(0)}\vert u\vert^{2(q-1)q^*}dx\leq \sigma.
\end{equation}
Since, $2(q-1)q^*>2$ and by Lemma \ref{lem21},
there exists $w \in L^{2(q-1)q^*}(B_r(0))$ such that $|u_{n_{j}}|\leq w$ a.e. in $B_r(0)$.\\
Thus, by \eqref{fnj}, for every $s\geq 1$ and a.e. $x\in B_r(0)$,
\begin{align*}
&\Big\vert f(\vert u_{n_{j}}\vert^{2})u_{n_{j}}-f(\vert u_{n_{j}}-\hat{u}_{j}\vert^{2})(u_{n_{j}}-\hat{u}_{j})- f(\vert
\hat{u}_{j}\vert^{2})\hat{u}_{j} \Big\vert^2\\
&\qquad\leq
C\left[
w^2 + |u|^2
+ (w + |u|)^{2(q-1)} (e^{\alpha(|u_{n_{j}}|+|u|)^2}-1)^2
\right].
\end{align*}
Then, by \eqref{bddintep}, applying the H\"older inequality,
\[
\int_{B_{r}(0)} (w + |u|)^{2(q-1)} (e^{\alpha(|u_{n_{j}}|+|u|)^2}-1)^2 dx
\leq
C\left( \int_{B_{r}(0)} (w + |u|)^{2(q-1)q^*} dx\right)^{1/q^*},
\]
and so
\[
w^2 + |u|^2
+ (w + |u|)^{2(q-1)} (e^{\alpha(|u_{n_{j}}|+|u|)^2}-1)^2
\in L^{1}(B_r(0)).
\]
Hence, applying the Lebesgue Dominated Convergence Theorem, we deduce that
\[
\int_{B_r(0)} \Big\vert f(\vert u_{n_{j}}\vert^{2})u_{n_{j}}-f(\vert u_{n_{j}}-\hat{u}_{j}\vert^{2})(u_{n_{j}}-\hat{u}_{j})- f(\vert
\hat{u}_{j}\vert^{2})\hat{u}_{j} \Big\vert^{2}dx \to 0
\text{ as } j\to +\infty.
\]
On the other hand, by \eqref{fnj}, \eqref{bddintep}, \eqref{Bjr2sigma}, \eqref{Bjrqsigma}, and \eqref{Brcsigma}, for $j$ large enough,
\begin{align*}
&\int_{B_{j}(0)\backslash B_{r}(0)} \Big\vert f(\vert u_{n_{j}}\vert^{2})u_{n_{j}}-f(\vert u_{n_{j}}-\hat{u}_{j}\vert^{2})(u_{n_{j}}-\hat{u}_{j})- f(\vert
\hat{u}_{j}\vert^{2})\hat{u}_{j} \Big\vert^{2}dx\\
&\qquad\leq
C\int_{B_{j}(0)\backslash B_{r}(0)} \Big(
\vert  u_{n_{j}}\vert^{2}
+\vert u\vert^{2}
+(\vert  u_{n_{j}}\vert+\vert  u\vert)^{2(q-1)}
(e^{2\alpha (\vert u_{n_{j}}\vert+\vert u\vert)^{2}}-1)\Big)dx\\
&\qquad\leq
C\left[\sigma
+\Big( \int_{B_{j}(0)\backslash B_{r}(0)} \vert  u_{n_{j}}\vert^{2(q-1)q^*}dx\Big)^{1/q^*}
+\Big( \int_{B_{j}(0)\backslash B_{r}(0)} \vert  u\vert^{2(q-1)q^*}dx\Big)^{1/q^*}
\right]\\
&\qquad\leq
C(\sigma +\sigma^{1/q^*}).
\end{align*}
Hence, since $\hat{u}_{j}=0$ in $B^{c}_{j}(0)$ for any $j\geq 1$, the arbitrariness of $\sigma>0$ allows to conclude.
\end{proof}

\section{The limit problem}\label{Secthree}

As it is usual for this type of questions, we consider the {\em limit} problem
\begin{align}
\label{3.2}
-\Delta u+\mu u=f(u^{2})u, \quad u:\mathbb{R}^{2}\to\mathbb{R}, \quad \mu>0.
\end{align}

We have already considered this problem in \cite[Section 2]{rDJ} for a particular value of $\mu$, namely for $\mu=V_0$. Here, due to the assumptions on $V$, we need some generalizations.

First of all we observe that solutions of \eqref{3.2} can be found as critical points of the $C^1$ functional
\begin{equation*}
I_{\mu}(u):=\frac{1}{2}\int_{\mathbb{R}^{2}}(\vert \nabla u\vert^{2}+\mu u^{2})dx -\frac{1}{2}\int_{\mathbb{R}^{2}}F(u^{2})dx
\end{equation*}
defined in $H^{1}(\mathbb{R}^{2}, \mathbb{R})$. Moreover, if
\begin{equation*}
\mathcal{N}_{\mu}:=\{u\in H^{1}(\mathbb{R}^{2}, \mathbb{R})\setminus\{0\}: I'_{\mu}(u)[u]=0\}
\quad\text{ and }\quad
c_{\mu}:=\inf_{u\in \mathcal{N}_{\mu}} I_{\mu}(u),
\end{equation*}
using (\ref{f1}) and (\ref{f4}), we have that for each fixed $u\in  H^{1}(\mathbb{R}^{2}, \mathbb{R})\backslash\{0\}$, there is a unique $t(u)>0$ such that
\begin{equation}\label{maxmu}
I_{\mu}(t(u)u)=\max_{t\geq 0}I_{\mu}(tu)\quad\text{and}\quad t(u)u\in \mathcal{N}_{\mu}.
\end{equation}
Then, arguing as in \cite[Lemma 4.1 and Theorem 4.2]{rW}, we also have that
\begin{equation}\label{cmu}
0<c_{\mu}=\inf_{u\in  H^{1}(\mathbb{R}^{2}, \mathbb{R})\backslash\{0\}} \max_{t\geq 0}I_{\mu}(tu).
\end{equation}
Observe finally that, arguing as in \cite[Lemma 2.3]{rDJ} where we consider only the case $\mu=V_0$, we can get that there exists a constant $K>0$ such that, for all $u\in \mathcal{N}_{\mu}$, $\| u\|_{\mu}\geq K$.

In particular, if $\mu\in(0,V^*]$, the assumptions on $f$ allow to get  the existence of a positive {\em ground state} solution $\omega_\mu$ of \eqref{3.2}, namely such that $I_{\mu}(\omega_\mu)\leq I_{\mu}(v)$ for all positive solutions $v \in H^1(\mathbb{R}^2,\mathbb{R})$ of \eqref{3.2}.\\

\begin{lemma}\label{gslimit}
For every $\mu\in(0,V^*]$, there exists a positive and radially symmetric ground state solution $\omega_\mu \in H^1(\mathbb{R}^2,\mathbb{R})$ for problem \eqref{3.2}.
\end{lemma}
\begin{proof}
We can argue exactly as in \cite[Lemma 2.2]{rDJ}, applying \cite[Corollary 1.5]{rASM} to the function $h(t):= f(t^2)t/\mu$. Here we observe only that the condition
\[
\text{there exist } \lambda>0 \text{ and } p>2\text{ such that } h(t)\geq \lambda t^{p-1}\text{ for } t\geq 0\text{ and }
\lambda>\left(\frac{p-2}{p}\right)^\frac{p-2}{2} S_p^{p/2}
\]
follows by (\ref{f4}), since
\[
h(t)\geq \frac{C_p}{\mu} t^{p-1}
\quad
\text{ for all } t\geq 0
\]
and
\[
\frac{C_p}{\mu}
>
\Big(\frac{8\theta}{\theta-2}\Big)^\frac{p-2}{2}
\Big(\frac{p-2}{p}\Big)^\frac{p-2}{2}
S_{p}^{p/2}
>
8^\frac{p-2}{2}
\Big(\frac{p-2}{p}\Big)^\frac{p-2}{2}
S_{p}^{p/2}
>
\Big(\frac{p-2}{p}\Big)^\frac{p-2}{2}
S_{p}^{p/2}.
\]
By standard arguments we can assume $\omega_\mu$ positive. We refer to \cite[Lemma 2.2]{rDJ} for the remaining details.
\end{proof}

Moreover the ground state $\omega_\mu$ decays exponentially at infinity with its gradient and belongs to $C^{2}(\mathbb{R}^{2}, \mathbb{R})\cap L^\infty(\mathbb{R}^{2}, \mathbb{R})$ (see also \cite[Proposition 2.1]{rZD}) and, for any $\mu_{1}, \mu_{2}\in(0,V^*]$ with $0<\mu_{1}<\mu_{2}$, we have
	\begin{equation}
		\label{moncmu}
	c_{\mu_1}
	=	I_{\mu_{1}} (\omega_{\mu_{1}})
	\leq \max_{t\geq 0} I_{\mu_{1}}(t\omega_{\mu_{2}})
	< \max_{t\geq 0} I_{\mu_{2}}(t\omega_{\mu_{2}})
	= I_{\mu_{2}}(\omega_{\mu_{2}})
	=c_{\mu_{2}}.
	\end{equation}

Let now $S_p>0$ be defined as in (\ref{f4}) and, for every $\mu>0$,
\[
S_{p}^\mu
:=\inf_{u\in H^{1}(\mathbb{R}^{2}, \mathbb{R})\backslash\{0\}}\frac{\Vert \nabla u\Vert_2^{2}+\mu\Vert u\Vert_2^{2}}{\Vert u\Vert_{p}^{2}}.
\]
We have
\begin{lemma}\label{Splemma}
$S_p^\mu=\mu^{2/p}S_p$.
\end{lemma}
\begin{proof}
Let $u\in H^{1}(\mathbb{R}^{2}, \mathbb{R})\backslash\{0\}$.\\
If $\tilde{u}:=u(\mu^{-1/2}\cdot)$, we have
\[
\frac{\Vert \nabla u \Vert_2^{2} + \mu \Vert u \Vert_2^{2}}{\Vert u \Vert_{p}^{2}}
=
\mu^{2/p} \frac{\Vert \nabla \tilde{u}\Vert_2^{2}+\Vert \tilde{u}\Vert_2^{2}}{\Vert \tilde{u}\Vert_{p}^{2}}
\geq \mu^{2/p} S_p
\]
and, passing to the infimum, we get $S_p^\mu\geq \mu^{2/p} S_p$.\\
On the other hand, if $\tilde{u}:=u(\mu^{1/2}\cdot)$, we have
\[
\frac{\Vert \nabla u \Vert_2^{2} + \Vert u \Vert_2^{2}}{\Vert u \Vert_{p}^{2}}
=
\mu^{-2/p} \frac{\Vert \nabla \tilde{u}\Vert_2^{2}+ \mu \Vert \tilde{u}\Vert_2^{2}}{\Vert \tilde{u}\Vert_{p}^{2}}
\geq \mu^{-2/p} S_p^\mu
\]
and, passing to the infimum, we get $S_p^\mu\leq \mu^{2/p} S_p$, concluding the proof.
\end{proof}
\begin{remark}\label{RemarkSp}
Observe that, by \cite[Theorem 1.35]{rW}, $S_p$ is achieved by a $C^2$ positive and radially symmetric function, and, arguing as in the proof of Lemma \ref{Splemma}, this holds for $S_p^\mu$ too.
\end{remark}

Now we give an upper estimate of the ground state energy for \eqref{3.2} with $\mu=V^*$, with $V^*$ defined in \eqref{SpV*}.
\begin{lemma}\label{le32}
The minimax level $c_{V^*}$ belongs to the interval $\displaystyle \Big(0, \frac{\theta-2}{16 \theta}\Big)$.
\end{lemma}
\begin{proof}
As observed in Remark \ref{RemarkSp}, there exists $\omega^{*}\in H^{1}(\mathbb{R}^{2}, \mathbb{R})\setminus\{0\}$ such that
\begin{equation*}
S_{p}^{V^*}=\frac{\Vert \nabla \omega^{*}\Vert_2^{2}+V^*\Vert  \omega^{*}\Vert_2^{2}}{\Vert \omega^{*}\Vert_{p}^{2}}.
\end{equation*}
Thus, using \eqref{cmu}, (\ref{f4}), and Lemma \ref{Splemma}, we have
\begin{align*}
0<c_{V^*}
&\leq \max_{t\geq 0}I_{V^*}(t\omega^{*})
\leq \max_{t\geq 0}\Big\{\frac{t^{2}}{2}(\Vert \nabla \omega^{*}\Vert_2^{2}+V^*\Vert  \omega^{*}\Vert_2^{2})
-\frac{C_{p}t^{p}}{p} \|\omega^{*}\|_p^{p}\Big\}\\
&\qquad
=\frac{p-2}{2p}\left[\frac{(S_{p}^{V^*})^{p}}{C_{p}^{2}}\right]^\frac{1}{p-2}
=\frac{p-2}{2p}\left[\frac{({V^*})^{2}S_{p}^p}{C_{p}^{2}}\right]^\frac{1}{p-2}
<\frac{\theta-2}{16 \theta}.
\end{align*}
\end{proof}

Finally we give the following result.

\begin{lemma}\label{lemFat}
Let $\mu\in (0, V^*]$, and $\{\omega_{n}\}\subset \mathcal{N}_{\mu}$ be a sequence of positive functions satisfying $I_{\mu}(\omega_{n})\rightarrow c_{\mu}$. Then $\{\omega_{n}\}$ is bounded
in  $H^{1}(\mathbb{R}^{2}, \mathbb{R})$ and, up to a subsequence, $\omega_{n}\rightharpoonup \omega$ in $H^{1}(\mathbb{R}^{2}, \mathbb{R})$.\\
Moreover we have that:
\begin{enumerate}
	\item if $\omega\neq 0$, then $\omega_{n}\rightarrow \omega\in \mathcal{N}_{\mu}$ in $H^{1}(\mathbb{R}^{2}, \mathbb{R})$ and $\omega$ is a ground state for \eqref{3.2};
	\item if $\omega= 0$, then there exists $\{\tilde{y}_{n}\}\subset \mathbb{R}^{2}$ with $\vert \tilde{y}_{n}\vert\rightarrow +\infty$ and $\tilde{\omega}\in \mathcal{N}_{\mu}$ such that, up to a subsequence, $\omega_{n}(\cdot+\tilde{y}_{n})\rightarrow \tilde{\omega}$ in $H^{1}(\mathbb{R}^{2}, \mathbb{R})$
	and $\tilde{\omega}$ is a ground state for \eqref{3.2}.
\end{enumerate}
\end{lemma}
\begin{proof}
By (\ref{f3}) and Lemma \ref{le32}, it follows that
 \begin{align*}
\frac{\theta-2}{2\theta} \limsup_n (\|\nabla \omega_n\|_{2}^2 + \mu \| \omega_n\|_{2}^2)
&\leq
\limsup_n\Big\{\Big(\frac{1}{2}-\frac{1}{\theta}\Big) (\|\nabla \omega_n\|_{2}^2 + \mu \| \omega_n\|_{2}^2)\\
&\qquad\qquad\qquad
+\int_{\mathbb{R}^{2}}\Big(\frac{1}{\theta}f( \omega_{n}^{2}) \omega_{n}^{2}-\frac{1}{2}F( \omega_{n}^{2})\Big)dx\Big\}\\
&=
\limsup_n \Big\{ I_{\mu}(\omega_{n})-\frac{1}{\theta}I'_{\mu}(\omega_{n})[\omega_{n}]\Big\}
=
c_{\mu}\leq c_{V^*}<\frac{\theta-2}{16\theta}.
\end{align*}
Thus,
\begin{equation}
\label{limsup<1/8}
\ell^2:=\limsup_n (\|\nabla \omega_n\|_{2}^2 + \mu \| \omega_n\|_{2}^2)<\frac{1}{8}
\end{equation}
and so there exists $\omega\in H^{1}(\mathbb{R}^{2}, \mathbb{R})$ such that, up to a subsequence, $\omega_{n}\rightharpoonup \omega$ in $H^{1}(\mathbb{R}^{2}, \mathbb{R})$. Moreover $\omega_{n}\rightarrow \omega$ in $L^{r}_{\rm loc}(\mathbb{R}^{2}, \mathbb{R})$, for any $r\geq 1$ and  $\omega_{n}\rightarrow \omega$ a.e. in $x\in \mathbb{R}^{2}$.\\
Now let us consider two different cases.\\
{\bf Case 1:} $\omega\neq 0$.\\
Observe that, for every $\phi\in C_{c}^{\infty}(\mathbb{R}^{2},\mathbb{R})$,
$$f(\omega_{n}^{2})\omega_{n}\phi\to f(\omega^{2})\omega\phi
\text{ a.e. in } \mathbb{R}^2 \text{ as }n\to+\infty$$
and that, by \eqref{2.8}, we have that for any $\zeta>0$, $q>2$, and $\alpha>4\pi$, there exists $C>0$ such that
\begin{align}\label{inequ}
	|f(\omega_{n}^{2})\omega_{n}\phi |
&\leq \zeta  \omega_{n}\vert \phi\vert
+C \omega_{n}^{q-1}(e^{\alpha \omega_{n}^{2}}-1)\vert \phi\vert
\leq
	\zeta  \omega_{n} \vert \phi\vert+C(\omega_{n}^{q}+\vert \phi\vert^{q})(e^{\alpha \omega_{n}^{2}}-1),
	\end{align}
with
\[
\zeta  \omega_{n} \vert \phi\vert+C(\omega_{n}^{q}+\vert \phi\vert^{q})(e^{\alpha \omega_{n}^{2}}-1)
\to
\zeta  \omega \vert \phi\vert+C(\omega^{q}+\vert \phi\vert^{q})(e^{\alpha \omega^{2}}-1)
\text{ a.e. in } \mathbb{R}^2 \text{ as }n\to+\infty.
\]
 Now, taking into account \eqref{limsup<1/8},  let us fix $\kappa\in(0,1-8\ell^2)$, and let
	\[
p^*\in\Big(1,\frac{1}{8\ell^2+\kappa}\Big),
\quad
	\alpha\in\Big(4\pi,\frac{4\pi}{p^*(8\ell^2+\kappa)}\Big).
	\]
If $\Omega:=\text{supp}(\phi)$, by \cite[Lemma A.1]{rW},
	there exists $\upsilon  \in L^{qq^*}(\Omega,\mathbb{R})$ such that $\omega_{n}\leq \upsilon$ a.e. in $\Omega$, where $qq^*>2$ and $q^*$ is the conjugate exponent of $p^*$ and, by \eqref{inequ}, we have
\begin{align*}
	|f(\omega_{n}^{2})\omega_{n}\phi |
	\leq
	\zeta \omega_{n} \vert \phi\vert+C (\upsilon^{q}+\vert \phi\vert^{q})(e^{\alpha \omega_{n}^{2}}-1)
	\to \zeta \omega \vert \phi\vert+C (\upsilon^{q}+\vert \phi\vert^{q})(e^{\alpha \omega^{2}}-1)
	\text{ a.e. in } \Omega.
	\end{align*}
Arguing as in the proof of Lemma \ref{2.0}, we have
		$$(e^{\alpha \omega_{n}^{2}}-1)\rightharpoonup (e^{\alpha \omega^{2}}-1) \text{ in } L^{p^*}(\mathbb{R}^{2}, \mathbb{R}).$$
		Moreover, since $ \omega_{n} \rightharpoonup  \omega$ in $L^{2}(\mathbb{R}^{2}, \mathbb{R})$
		we have that
		\[
		\zeta\int_{\mathbb{R}^{2}} \omega_{n} \vert \phi\vert dx
		+C\int_{\mathbb{R}^{2}}(\upsilon^{q}+\vert \phi\vert^{q})(e^{\alpha \omega_{n}^{2}}-1) dx
		\to
		\zeta\int_{\mathbb{R}^{2}} \omega\vert \phi\vert dx
		+C \int_{\mathbb{R}^{2}}(\upsilon^{q}+\vert \phi\vert^{q})(e^{\alpha \omega^{2}}-1)dx
		\]
		as $n\to +\infty$.\\
		Hence, a variant of the Lebesgue Dominated Convergence Theorem
		implies that
		$$
		\int_{\mathbb{R}^{2}}f(\omega_{n}^{2})\omega_{n}\phi dx\rightarrow \int_{\mathbb{R}^{2}}f(\omega^{2})\omega\phi dx,
		$$
		and so $\omega$ is a nontrivial critical point for $I_{\mu}$.\\
Since,  by (\ref{f3}) and the Fatou's Lemma,
$$
\int_{\mathbb{R}^{2}}\Big(\frac{1}{\theta}f(\omega^{2})\omega^{2}-\frac{1}{2}F( \omega^{2})\Big)dx
\leq
\liminf_n\int_{\mathbb{R}^{2}}\Big(\frac{1}{\theta}f(\omega_{n}^{2}) \omega_{n}^{2}-\frac{1}{2}F(\omega_{n}^{2})\Big)dx,$$
we have
 \begin{align*}
c_{\mu}
&\leq
I_{\mu}(\omega)
=I_{\mu}(\omega)-\frac{1}{\theta}I'_{\mu}(\omega)[\omega]\\
&=
\Big(\frac{1}{2}-\frac{1}{\theta}\Big) (\|\nabla \omega\|_{2}^2 + \mu \| \omega\|_{2}^2)
+\int_{\mathbb{R}^{2}}\Big(\frac{1}{\theta}f( \omega^{2})\omega^{2}-\frac{1}{2}F( \omega^{2})\Big)dx\\
&\leq
\liminf_n\Big\{\Big(\frac{1}{2}-\frac{1}{\theta}\Big) (\|\nabla \omega_n\|_{2}^2 + \mu \| \omega_n\|_{2}^2)
+\int_{\mathbb{R}^{2}}\Big(\frac{1}{\theta}f( \omega_{n}^{2}) \omega_{n}^{2}-\frac{1}{2}F( \omega_{n}^{2})\Big)dx\Big\}\\
&=
\liminf_n \Big\{ I_{\mu}(\omega_{n})-\frac{1}{\theta}I'_{\mu}(\omega_{n})[\omega_{n}]\Big\}
=c_{\mu}
\end{align*}
and, by the assumptions,
 \begin{align*}
0
&\leq
\liminf_n \left[\Big(\frac{1}{2}-\frac{1}{\theta}\Big)
(\|\nabla \omega_n\|_{2}^2
+ \mu \| \omega_n\|_{2}^2
-\|\nabla \omega\|_{2}^2
-\mu \| \omega\|_{2}^2)\right]\\
&
\leq
\limsup_n \left[\Big(\frac{1}{2}-\frac{1}{\theta}\Big)
(\|\nabla \omega_n\|_{2}^2
+ \mu \| \omega_n\|_{2}^2
-\|\nabla \omega\|_{2}^2
-\mu \| \omega\|_{2}^2)\right]
\\
&=
\limsup_n \Big[
I_{\mu}(\omega_n) - c_{\mu}
+\int_{\mathbb{R}^{2}}\Big(\frac{1}{\theta}f(\omega^{2})\omega^{2}-\frac{1}{2}F( \omega^{2})\Big)dx
-\int_{\mathbb{R}^{2}}\Big(\frac{1}{\theta}f(\omega_{n}^{2}) \omega_{n}^{2}-\frac{1}{2}F( \omega_{n}^{2})\Big)dx
\Big]
\\
&=
\int_{\mathbb{R}^{2}}\Big(\frac{1}{\theta}f(\omega^{2}) \omega^{2}-\frac{1}{2}F( \omega^{2})\Big)dx
-\liminf_n \Big[\int_{\mathbb{R}^{2}}\Big(\frac{1}{\theta}f(\omega_{n}^{2}) \omega_{n}^{2}-\frac{1}{2}F( \omega_{n}^{2})\Big)dx\Big]
\leq 0.
\end{align*}
{\bf Case 2:} $\omega= 0$.\\
In this case, there exist $R, \eta>0$, and $(\tilde{y}_{n})\subset \mathbb{R}^{2}$ such that
\begin{equation}\label{4.7}
\lim_n\int_{B_{R}(\tilde{y}_{n})}\omega_{n}^{2}dx\geq \eta.
\end{equation}
Indeed, if this does not hold, for any $R>0$, one has
\begin{equation*}
\lim_n \sup_{y\in \mathbb{R}^{2}}\int_{B_{R}(y)} \omega_{n}^{2}dx=0.
\end{equation*}
Thus, by \cite[Chapter 6, Lemma 8.4]{rK}, for every $\tau>2$,
\begin{equation*}
\lim_n \Vert \omega_{n}\Vert_{\tau}=0.
\end{equation*}
Moreover, by \eqref{2.10}, \eqref{ineqe}, Lemma \ref{le24}, and \eqref{limsup<1/8}, and arguing again as in Lemma \ref{2.0},
for $0<\zeta<\mu/2$, $q>2$ and $\alpha>4\pi$, there exists $C>0$ such that
\begin{align*}
\| \nabla \omega_{n}\|_2^{2}+\mu \|\omega_{n}\|_2^{2}
&\leq
\zeta \|\omega_{n}\|_2^{2}
+C\int_{\mathbb{R}^{2}} \omega_{n}^{q}(e^{\alpha   \omega_{n}^{2}}-1)dx\\
&\leq
\zeta \|\omega_{n}\|_2^{2}
+ C \Vert \omega_{n}\Vert_{qq^{*}}^{q}\Big(\int_{\mathbb{R}^{2}}(e^{p^{*}\alpha \omega_{n}^{2}}-1)dx\Big)^{1/p^{*}}\\
&\leq
\zeta \|\omega_{n}\|_2^{2}
+ C \Vert \omega_{n}\Vert_{qq^{*}}^{q}
=\zeta \|\omega_{n}\|_2^{2}
+ o_{n}(1).
\end{align*}
Hence $\omega_{n} \rightarrow 0 $ in $H^1(\mathbb{R}^2,\mathbb{R})$ as $n\rightarrow +\infty$, and so $I_{\mu}(\omega_{n})\rightarrow 0$ as $n\rightarrow +\infty$. But this is in contradiction with $I_{\mu}(\omega_{n})\rightarrow c_{\mu}>0$ as $n\rightarrow +\infty$.\\
Now we claim that $\vert \tilde{y}_{n}\vert\rightarrow+\infty$.\\
Otherwise, there exists $\bar{R}>0$ such that
\begin{equation*}
\lim_n\int_{B_{\bar{R}}(0)}\omega_{n}^{2}dx\geq \eta
\end{equation*}
and so $\omega\neq 0$.\\
Moreover $\omega_{n}(\cdot+\tilde{y}_{n})\in \mathcal{N}_{\mu}$ and, the invariance of $I_{\mu}$ and $\Vert \nabla \cdot \Vert_2^2 + \mu \|\cdot\|_2^2$ by translations implies
\[
I_{\mu}(\omega_{n}(\cdot+\tilde{y}_{n}))\rightarrow c_{\mu}
\text{ and }
\limsup_n (\|\nabla \omega_n (\cdot+\tilde{y}_{n})\|_{2}^2 + \mu \| \omega_n (\cdot+\tilde{y}_{n})\|_{2}^2) <\frac{1}{8}.
\]
Thus, arguing as before, there exists $\tilde{\omega}\in H^{1}(\mathbb{R}^{2}, \mathbb{R})$ which, by \eqref{4.7}, is nontrivial, such that
\begin{align*}
 \omega_{n}(\cdot+\tilde{y}_{n})\rightharpoonup \tilde{\omega}\quad \text{in}\quad H^{1}(\mathbb{R}^{2}, \mathbb{R}).
\end{align*}
Hence, repeating the same arguments of Case 1, we conclude.
\end{proof}

\section{Proof of Theorem \ref{gt62} }\label{Sec3}
The goal of this section is to prove Theorem \ref{gt62}.

We say that $u\in H_{\varepsilon}$ is a weak solution of \eqref{1.6}, if for any $\phi\in H_{\varepsilon}$
\begin{equation*}
\text{Re}\int_{\mathbb{R}^{2}}(\nabla_{A_{\varepsilon}}u
\overline{\nabla_{A_{\varepsilon}}\phi}+V_{\varepsilon}(x)u\overline{\phi})dx=\text{Re}\int_{\mathbb{R}^{2}}f(\vert u\vert^{2})u\overline{\phi}dx.
\end{equation*}
Such solutions can be found as the critical points of the $C^1$ functional $J_{\varepsilon}: H_{\varepsilon}\rightarrow \mathbb{R}$ given by
\begin{equation*}
J_{\varepsilon}(u)=\frac{1}{2}\int_{\mathbb{R}^{2}}(\vert \nabla_{A_\varepsilon}u\vert^{2}+V_{\varepsilon}(x)\vert u\vert^{2})dx-\frac{1}{2}\int_{\mathbb{R}^{2}}F(\vert
u\vert^{2})dx.
\end{equation*}

The functional $J_{\varepsilon}$ satisfies the Mountain Pass geometry, namely we have
\begin{lemma}\label{mountain}
For any $\varepsilon>0$,
\begin{enumerate}[label=(\roman{*}), ref=\roman{*}]
	\item \label{MPGi} there exist $\beta, r>0$ such that $J_{\varepsilon}(u)\geq \beta$ if $\Vert u\Vert_{\varepsilon}=r$;
	\item \label{MPGii} there exists $e\in H_{\varepsilon}$ with $\Vert e\Vert_{\varepsilon}>r$ such that $J_{\varepsilon}(e)<0$.
\end{enumerate}
\end{lemma}
\begin{proof}
By \eqref{2.11}, the H\"{o}lder and Sobolev inequalities, and \eqref{ineqe}, fixed $q>2$, for any $\zeta>0$ and $\alpha>4\pi$, there exists a constant $C>0$, which depends on $q$, $\alpha$, $\zeta$, such that
\begin{equation}
\label{3.1}
\begin{split}
\int_{\mathbb{R}^2}F(\vert u\vert^{2}) dx
&\leq
\zeta \| u\|_2^{2}
+C \int_{\mathbb{R}^{2}}\vert u\vert^{q}(e^{\alpha \vert u\vert^{2}}-1)dx\\
&\leq
\zeta \| u\|_2^{2}
+C\Vert u\Vert_{2q}^{q}\Big(\int_{\mathbb{R}^{2}}(e^{\alpha \vert u\vert^{2}}-1)^{2}dx\Big)^{1/2}\\
&\leq
\zeta \| u\|_2^{2}
+ C\Vert u\Vert_{\varepsilon}^{q}\Big(\int_{\mathbb{R}^{2}}(e^{2\alpha \vert u\vert^{2}}-1)dx\Big)^{1/2}.
\end{split}
\end{equation}
Moreover, by the diamagnetic inequality \eqref{2.1} and \eqref{V}, for any $u\in H_\varepsilon\setminus\{0\}$,
$$\frac{\vert u\vert}{\Vert u\Vert_{\varepsilon}}\in H^{1}(\mathbb{R}^{2}, \mathbb{R}),
\quad
\Big\Vert \frac{\vert u\vert}{\Vert u\Vert_{\varepsilon}}\Big\Vert_2^2 \leq \frac{1}{V_0},
\quad
\Big\Vert \nabla \frac{ \vert u\vert}{\Vert u\Vert_{\varepsilon}}\Big\Vert_2^2\leq 1.$$
Thus, for every $u \in H_\varepsilon$ with $\Vert u\Vert_{\varepsilon}=r>0$ and $r^2< \pi/\alpha$, using Lemma \ref{le24}, we get that there exists $C>0$ independent of $r$ such that
\begin{equation}
\label{3.1ter}
\int_{\mathbb{R}^{2}}(e^{2\alpha \vert u\vert^{2}}-1)dx
=\int_{\mathbb{R}^{2}}(e^{2\alpha r^{2}\big(\frac{\vert u\vert}{\Vert u\Vert_{\varepsilon}}\big)^{2}}-1)dx
<\int_{\mathbb{R}^{2}}(e^{2\pi\big(\frac{\vert u\vert}{\Vert u\Vert_{\varepsilon}}\big)^{2}}-1)dx
\leq C.
\end{equation}
Hence, by \eqref{3.1} and \eqref{3.1ter}, we get that for any $\zeta>0$, there exits $C>0$ independent of $r$ such that
\[
J_{\varepsilon}(u)
\geq
\frac{1}{2} r^2
- \frac{\zeta}{V_0} \int_{\mathbb{R}^{2}} V_\varepsilon(x) | u|^{2} dx
- Cr^{q}
\geq\frac{1}{2}\Big(1-\frac{\zeta}{V_{0}}\Big)r^{2}-C r^{q}
\]
and so, since $q>2$, for $r$ small enough we obtain (\ref{MPGi}).\\
Finally, property (\ref{MPGii}) is an easy consequence of (\ref{f4}), since, fixed $\varphi\in C_c^\infty(\mathbb{R}^2,\mathbb{C})\setminus\{0\}$, we have that
\[
J_{\varepsilon}(t\varphi)\leq\frac{t^{2}}{2} \Vert \varphi \Vert^{2}_{\varepsilon}- \frac{C_{p}}{p}t^{p} \| \varphi\|_p^{p}.
\]
\end{proof}

Let now
\begin{equation}\label{ceps}
c_{\varepsilon}:=\inf_{\gamma\in \Gamma_{\varepsilon}}\max_{t\in [0, 1]}J_{\varepsilon}(\gamma(t)),
\end{equation}
where
\begin{equation*}
\Gamma_{\varepsilon}=\{\gamma\in C([0, 1], H_{\varepsilon}): \gamma(0)=0 \text{ and } J_{\varepsilon}(\gamma(1))< 0\}.
\end{equation*}
To find ground states and multiple solutions for \eqref{1.1}, as usual, we consider the Nehari manifold
\begin{equation*}
\mathcal{N}_{\varepsilon}:=\{u\in H_{\varepsilon}\backslash\{0\}: J'_{\varepsilon}(u)[u]=0\}.
\end{equation*}
In virtue of (\ref{f4}), we can show that for any $u\in H_\varepsilon\setminus\{0\}$ there exists a unique $t_\varepsilon>0$ such that $t_{\varepsilon}u\in \mathcal{N}_{\varepsilon}$ and
\begin{equation}
\label{Nehepsmax}
J_{\varepsilon}(t_{\varepsilon}u)=\max_{t\geq 0}J_{\varepsilon}(tu)
\end{equation}
and, using classical arguments (see e.g. \cite[Lemma 4.1 and Theorem 4.2]{rW}),
\begin{equation}\label{minmaxNehari}
c_{\varepsilon}=\inf_{u\in H_{\varepsilon}\backslash\{0\}}\sup_{t\geq 0}J_{\varepsilon}(tu)=\inf_{u\in \mathcal{N}_{\varepsilon}}J_{\varepsilon}(u).
\end{equation}
Moreover we have the following useful preliminary property.
\begin{lemma}\label{unifbound}
There exist $K_1,K_2>0$ such that for every $\varepsilon>0$ and for every $u\in \mathcal{N}_{\varepsilon}$
\[
\|u\|_\varepsilon\geq K_1
\text{ and }
J_\varepsilon(u)\geq K_2.
\]
\end{lemma}
\begin{proof}
Observe that, for every $\varepsilon>0$, if $u\in \mathcal{N}_{\varepsilon}$, by \eqref{2.10}, fixed $q>2$, for any $\zeta>0$ and $\alpha>4\pi$, there exists a constant $C>0$, which depends on $q$, $\alpha$, $\zeta$, such that
\begin{equation}\label{IneqNehf}
\| u\|_\varepsilon^2
= \int_{\mathbb{R}^{2}} f(|u|^2)|u|^2\leq \zeta \|u\|_2^2
+ C \int_{\mathbb{R}^{2}}\vert u\vert^{q}(e^{\alpha \vert u\vert^{2}}-1)dx.
\end{equation}
In particular, for $\zeta<V_0/2$, by \eqref{V}, we have that
\begin{equation}\label{2minusz}
\| u \|_\varepsilon^2-\zeta \|u\|_2^2
\geq
\frac{1}{2} \| u \|_\varepsilon^2
+ \frac{1}{2} \int_{\mathbb{R}^{2}}[ V_\varepsilon(x)-V_0]| u |^{2}dx\\
\geq
\frac{1}{2} \| u \|_\varepsilon^2.
\end{equation}
By the H\"{o}lder inequality and \eqref{ineqe}, we infer that
\begin{equation}\label{bddq}
\int_{\mathbb{R}^{2}} | u |^{q}(e^{\alpha | u |^{2}}-1)dx
\leq
\Vert u\Vert_{2q}^{q}\Big(\int_{\mathbb{R}^{2}}(e^{\alpha \vert u\vert^{2}}-1)^{2}dx\Big)^{1/2}
\leq
C\Vert u\Vert_{\varepsilon}^{q}\Big(\int_{\mathbb{R}^{2}}(e^{2\alpha \vert u\vert^{2}}-1)dx\Big)^{1/2}.
\end{equation}
Now, assume by contradiction that for every $n\in\mathbb{N}$ there exist $\varepsilon_n>0$ and $u_n \in \mathcal{N}_{\varepsilon_n}$ such that $\|u_n\|_{\varepsilon_n} \to 0$ as $n\to +\infty$. Then, using the diamagnetic inequality \eqref{2.1} and \eqref{V}, also $\| \nabla |u_n| \|_2^2 + V_0 \| u_n\|_2^2 \to 0$ as $n\to +\infty$ and so, for $n$ large enough, using Lemma \ref{le24}, for some $\bar{\alpha}\in (0,4\pi)$,
\begin{equation}\label{bddC}
\begin{split}
\int_{\mathbb{R}^{2}}(e^{2\alpha \vert u_n\vert^{2}}-1)dx
&=
\int_{\mathbb{R}^{2}}(e^{2\alpha (\| \nabla |u_n| \|_2^2 + V_0 \| u_n\|_2^2) \left( \frac{|u_n|}{\sqrt{\| \nabla |u_n| \|_2^2 + V_0 \| u_n\|_2^2}}\right)^2} -1)dx\\
&\leq
\int_{\mathbb{R}^{2}}(e^{\bar{\alpha} \left( \frac{|u_n|}{\sqrt{\| \nabla |u_n| \|_2^2 + V_0 \| u_n\|_2^2}}\right)^2} -1)dx
\leq C.
\end{split}
\end{equation}
Thus, by \eqref{IneqNehf}, \eqref{2minusz}, \eqref{bddq}, and \eqref{bddC}, we get
\begin{equation*}
\frac{1}{2}\| u_n \|_{\varepsilon_n}^2
\leq  C \| u_n \|_{\varepsilon_n}^q
\end{equation*}
reach a contradiction.\\
The second part is an immediate consequence of the first one observing that, if $u\in \mathcal{N}_{\varepsilon}$, by (\ref{f3}),
\[
J_\varepsilon(u)
=
J_\varepsilon(u)-\frac{1}{\theta}J_\varepsilon'(u)[u]
=\frac{\theta-2}{2\theta} \| u \|_{\varepsilon}^2
+\frac{1}{\theta} \int_{\mathbb{R}^{2}} \Big[ f(|u|^2)|u|^2-\frac{\theta}{2}F(|u|^2)\Big]dx
\geq
\frac{\theta-2}{2\theta} \| u \|_{\varepsilon}^2.
\]
\end{proof}
The functional $J_\varepsilon$ satisfies the following fundamental property.
\begin{lemma}\label{PSbdd}
If $\{u_{n}\} \subset H_{\varepsilon}$ is a $(PS)$ sequence for $J_{\varepsilon}$, then it is bounded in $H_{\varepsilon}$.
\end{lemma}
\begin{proof}
Let $d:=\lim_n J_{\varepsilon}(u_n)\in\mathbb{R}$.
Using (\ref{f3}) we have that
\begin{align*}
\frac{\theta-2}{2\theta} \|u_n\|_\varepsilon^2
&\leq
(\frac{1}{2}-\frac{1}{\theta})\|u_n\|_\varepsilon^2
+\int_{\mathbb{R}^{2}} \Big[\frac{1}{\theta} f(|u_n|^2)|u_n|^2-\frac{1}{2}F(|u_n|^2)\Big]dx\\
&=
J_{\varepsilon}(u_n) - \frac{1}{\theta} J_{\varepsilon}'(u_n)[u_n]\\
&\leq
d+o_n(1) + o_n(1)  \|u_n\|_\varepsilon
\end{align*}
and we can conclude.
\end{proof}

The following lemma is very important for the proof of Theorem \ref{gt62}.
	\begin{lemma}\label{newadd}
		There exists $\bar{\varepsilon}>0$ such that, for all $\varepsilon\in (0, \bar{\varepsilon})$, $c_{\varepsilon}\leq c_{V_0}$.
	\end{lemma}
	\begin{proof}
		Assume for simplicity that $V(0)=V_0$ and  let $\omega\in H^{1}(\mathbb{R}^{2}, \mathbb{R})$ be a
		positive ground state solution of the limit problem \eqref{3.2} with $\mu=V_{0}$ (see Lemma \ref{gslimit}) and $\eta\in C_{c}^{\infty}(\mathbb{R}^{2}, [0, 1])$ be a cut-off function such that
		$\eta=1$ in $B_{1}(0)$ and $\eta=0$ in $B^{c}_{2}(0)$. Moreover, let
		$$\omega_{r}(x):=\eta_{r}(x)\omega(x)e^{iA(0)\cdot x},
		\qquad
		\eta_{r}(x)=\eta(x/r)
		\text{ for }  r>0.$$
		Observe that $\vert \omega_{r}\vert=\eta_{r}\omega$ and $\omega_{r}\in H_{\varepsilon}$ by Lemma \ref{gt520}.\\
		Let $\varepsilon_{n}\rightarrow 0^+$ as $n\rightarrow+\infty$. Taking $t_{\varepsilon_{n},
			r}>0$ such that
		$t_{\varepsilon_{n}, r}\omega_{r}\in \mathcal{N}_{\varepsilon_{n}}$, by \eqref{minmaxNehari} we have
		\begin{equation}\label{small}
			c_{\varepsilon_{n}}
			\leq
			J_{\varepsilon_{n}}(t_{\varepsilon_{n}, r}\omega_{r})
			=\frac{t^{2}_{\varepsilon_{n}, r}}{2}\int_{\mathbb{R}^{2}}\vert \nabla_{A_{\varepsilon_{n}}} \omega_{r}\vert^{2}dx
			+\frac{t^{2}_{\varepsilon_{n}, r}}{2}\int_{\mathbb{R}^{2}}V_{\varepsilon_{n}}( x)\vert \omega_{r}\vert^{2}dx
			-\frac{1}{2}\int_{\mathbb{R}^{2}}F(t^{2}_{\varepsilon_{n}, r}\vert \omega_{r}\vert^{2})dx.
		\end{equation}
		Observe that, for any fixed $r>0$,
		\begin{align}\label{small1}
			\lim_{n}\int_{\mathbb{R}^{2}}\vert \nabla_{A_{\varepsilon_{n}}}\omega_{r}\vert^{2}dx=\int_{\mathbb{R}^{2}}\vert \nabla(\eta_{r}\omega)\vert^{2}dx.
		\end{align}
		Indeed,
		\begin{align*}
			\int_{\mathbb{R}^{2}}\vert \nabla_{A_{\varepsilon_{n}}}\omega_{r}\vert^{2}dx
			&=
			\int_{\mathbb{R}^{2}}\Big|\frac{\nabla}{i}(\eta_{r}\omega)-[A_{\varepsilon_n}(x)-A(0)]\eta_{r}\omega\Big|^2dx
			\\
			&=
			\int_{\mathbb{R}^{2}}\vert \nabla(\eta_{r}\omega)\vert^{2}dx+\int_{\mathbb{R}^{2}}\vert A_{\varepsilon_{n}}( x)-A(0)\vert^{2}\eta_{r}^2\omega^2 dx
		\end{align*}
	and
	\[
	\int_{\mathbb{R}^{2}}\vert A_{\varepsilon_{n}}( x)-A(0)\vert^{2}\eta_{r}^2\omega^2 dx
	\leq
	\int_{B_{2r}(0)}\vert A_{\varepsilon_{n}}( x)-A(0)\vert^{2}\omega^2dx
	=o_n(1).
	\]
	Moreover
	\begin{equation}
	\label{small2b}
	\lim_n \int_{\mathbb{R}^{2}}V_{\varepsilon_{n}}( x)\vert\omega_{r}\vert^{2}dx
	=
	V_0\int_{\mathbb{R}^{2}}\eta_{r}^2\omega^{2}dx
	\end{equation}
	since
	\[
	\Big|  \int_{\mathbb{R}^{2}}V_{\varepsilon_{n}}( x)\vert\omega_{r}\vert^{2}dx
	-V_0\int_{\mathbb{R}^{2}}\eta_{r}^2\omega^{2}dx\Big|
	\leq
	\int_{\mathbb{R}^{2}}|V_{\varepsilon_{n}}( x) -V_0|\eta_{r}^2\omega^{2}dx
	\leq
	\int_{B_{2r}(0)} |V_{\varepsilon_{n}}( x) -V_0|\omega^{2}dx
	=o_n(1).
	\]
		We claim that, if $r>0$ is fixed, $\{t_{\varepsilon_{n}, r}\}$ is bounded.\\
		Indeed, let us assume by contradiction that there exists a subsequence, still denoted by $t_{\varepsilon_{n}, r}$, such that $t_{\varepsilon_{n}, r}\to+\infty$ as $n\rightarrow +\infty$.\\
		Since $J'_{\varepsilon_{n}}(t_{\varepsilon_{n}, r}\omega_{r}) [t_{\varepsilon_{n}, r}\omega_{r}]=0$,
		by (\ref{f1}) and (\ref{f4}) we have
		\[
		\int_{\mathbb{R}^{2}}\vert \nabla_{A_{\varepsilon_{n}}} \omega_{r}\vert^{2}dx
		+\int_{\mathbb{R}^{2}}V_{\varepsilon_{n}}( x)\vert
		\omega_{r}\vert^{2}dx
		=\int_{\mathbb{R}^{2}}f(t^{2}_{\varepsilon_{n}, r}\vert \omega_{r}\vert^{2})\vert \omega_{r}\vert^{2}dx
		\geq \int_{B_{r}(0)}f(t^{2}_{\varepsilon_{n}, r}\vert \omega\vert^{2})\vert \omega\vert^{2}dx
		\geq
		Ct^{p-2}_{\varepsilon_{n}, r}
		\to
		+\infty
		\]
		as $n\rightarrow +\infty$, which is a contradiction with \eqref{small1} and \eqref{small2b}.\\
		Thus, up to a subsequence, there exists $t_{r}\geq0$ such that $t_{\varepsilon_{n}, r}\rightarrow t_{r}$ as $n\rightarrow +\infty$ for $r$ fixed.\\		
		Therefore, by \eqref{small},
		\begin{equation}
		\label{cen}
		\limsup_n c_{\varepsilon_{n}}
		\leq \frac{t^{2}_{r}}{2}\int_{\mathbb{R}^{2}}\vert \nabla(\eta_{r}\omega)\vert^{2}dx
		+\frac{t^{2}_{r}}{2} V_0\int_{\mathbb{R}^{2}}\vert \eta_{r}\omega\vert^{2}dx
		-\frac{1}{2}\int_{\mathbb{R}^{2}}F(t^{2}_{r}\vert\eta_{r}\omega\vert^{2})dx
		=I_{V_0}(t_{r}\eta_{r}\omega).
		\end{equation}
	Let us consider now $r_{m}\rightarrow +\infty$ as $m\rightarrow+\infty$.\\
		Since
			\[
			\int_{\mathbb{R}^{2}}\vert \nabla (\eta_{r_{m}}\omega-\omega)\vert^{2}dx
			=
			\int_{\mathbb{R}^{2}}\vert \nabla \eta_{r_{m}}\vert^{2}\omega^2 dx
			+2\int_{\mathbb{R}^{2}}(\eta_{r_{m}}-1) \omega\nabla \eta_{r_{m}}\nabla\omega dx
			+\int_{\mathbb{R}^{2}}(\eta_{r_{m}}-1)^2 \vert \nabla \omega\vert^{2}dx,
			\]
\[
\int_{\mathbb{R}^{2}}\vert \nabla \eta_{r_{m}}\vert^{2}\omega^2 dx
\leq
\frac{C}{r_m^2} \int_{B_{r_m}^c}\omega^2 dx
=o_m(1),
\]
\[
\Big|\int_{\mathbb{R}^{2}}(\eta_{r_{m}}-1) \omega\nabla \eta_{r_{m}}\nabla\omega dx\Big|
\leq
\frac{C}{r_m}\Big( 	\int_{B_{r_m}^c}\omega^2 dx 	\Big)^{1/2}
\Big( \int_{B_{r_m}^c}|\nabla\omega|^2 dx \Big)^{1/2}
=o_m(1),
\]
\[
\int_{\mathbb{R}^{2}}(\eta_{r_{m}}-1)^2 \vert \nabla \omega\vert^{2}dx
\leq
\int_{B_{r_m}^c}|\nabla\omega|^2 dx
=o_{m}(1),
\]
\[
\int_{\mathbb{R}^{2}}(\eta_{r_{m}}-1)^2\omega^{2}dx
\leq
\int_{B_{r_m}^c}\omega^2 dx
=o_{m}(1),
\]
then $\eta_{r_{m}}\omega\rightarrow \omega$ in $H^{1}(\mathbb{R}^{2}, \mathbb{R})$ as $m\rightarrow +\infty$.\\
Now, let $\tau_{r_{m}}>0$ such that $\tau_{r_{m}}\eta_{r_{m}}\omega\in \mathcal{N}_{V_{0}}$.\\
Arguing as before, since
\begin{equation}
\label{Nehtrn}
\int_{\mathbb{R}^{2}}\vert \nabla (\eta_{r_{m}}\omega)\vert^{2}dx
+V_{0}\int_{\mathbb{R}^{2}} \eta_{r_{m}}^2\omega^{2}dx
=\int_{\mathbb{R}^{2}}f(\tau^{2}_{r_{m}}\vert\eta_{r_{m}}\omega\vert^{2})\eta_{r_{m}}^2\omega^{2}dx,
\end{equation}
we have that $\{\tau_{r_{m}}\}$ is bounded.\\
Thus, up to a subsequence, $\tau_{r_{m}}\rightarrow \bar{\tau}\geq 0$ as $m\rightarrow +\infty$ and, by \eqref{Nehtrn}, the uniqueness of the projection in $\mathcal{N}_{V_{0}}$, and since $\omega\in \mathcal{N}_{V_{0}}$ we get that $\bar{\tau}=1$.\\
Hence, by \eqref{cen} and \eqref{maxmu},
\[
\limsup_n c_{\varepsilon_{n}}
\leq I_{V_0}(t_{r_m}\eta_{r_m}\omega)
\leq I_{V_0}(\tau_{r_m}\eta_{r_m}\omega)
\]
and so, passing to the limit as $m\to+\infty$, we get
\[
\limsup_n c_{\varepsilon_{n}}
\leq
I_{V_{0}}(\omega)=c_{V_{0}}.
\]
\end{proof}

In our arguments we will use the following result that is a version of the celebrated Lions Lemma (see e.g. \cite{rW} and \cite[Lemma 3.3]{rDJ}).

\begin{lemma}\label{le35}
Let $\{u_{n}\} \subset H_{\varepsilon}$ satisfy $J_\varepsilon'(u_n)[u_n]\to 0$, $u_{n}\rightharpoonup 0$ in $H_{\varepsilon}$, as $n\to +\infty$,
and $\ell^2:=\limsup_{n}\| \nabla |u_n| \|_2^2  <1 $.  Then, one of the following alternatives occurs:
\begin{enumerate}[label=(\roman{*}), ref=\roman{*}]
	\item \label{le35i}$u_{n}\rightarrow 0$ in $H_{\varepsilon}$ as $n\to +\infty$;
	\item \label{le35ii} there are a sequence $\{y_{n}\} \subset \mathbb{R}^{2}$ and constants $R$, $\beta>0$ such that
	\begin{equation*}
	\liminf_n \int_{B_{R}(y_{n})}\vert u_{n}\vert^{2}dx\geq \beta.
	\end{equation*}
\end{enumerate}
\end{lemma}
	\begin{proof}
		If (\ref{le35ii}) does not hold, then, for every $R>0$,
		\begin{equation*}
		\lim_n \sup_{y\in \mathbb{R}^{2}}\int_{B_{R}(y)}\vert u_{n}\vert^{2}dx=0.
		\end{equation*}
		Since $\{|u_{n}|\}$ is bounded in $H^1(\mathbb{R}^2)$,  \cite[Chapter 6, Lemma 8.4]{rK} implies that
		\begin{equation}
		\label{kes}
		\| u_{n}\|_\tau\rightarrow 0
		\text{ as }
		n\rightarrow+\infty
		\text{ for all }
		\tau>2.
		\end{equation}
		Moreover, using \eqref{2.10},  we have that, fixed $q>2$, for any $\zeta>0$ and $\alpha>4\pi$, there exists $C>0$ such that
		\begin{equation}
		\label{zetaC}
		\begin{split}
		0\leq \Vert u_{n}\Vert_{\varepsilon}^{2}
		&=
		\int_{\mathbb{R}^{2}}f(\vert u_{n}\vert^{2})\vert u_{n}\vert^{2}dx+o_{n}(1)\\
		&\leq \zeta \Vert u_{n}\Vert_2^{2}
		+C\int_{\mathbb{R}^{2}}\vert u_{n}\vert^{q}(e^{\alpha\vert u_{n}\vert^{2}}-1)dx+o_{n}(1)\\
		&\leq \frac{\zeta}{V_{0}}\Vert u_{n}\Vert_{\varepsilon}^{2}+C\int_{\mathbb{R}^{2}}\vert u_{n}\vert^{q}(e^{\alpha\vert u_{n}\vert^{2}}-1)dx+o_{n}(1).
		\end{split}
		\end{equation}
Let us fix $\kappa\in(0,1-\ell^2)$ and take
	\[
	r\in\Big(1,\frac{1}{\ell^2+\kappa}\Big),
	\qquad
	\alpha\in\Big(4\pi,\frac{4\pi}{r(\ell^2+\kappa)}\Big).
	\]
		Thus, the H\"{o}lder inequality, \eqref{ineqe}, and Lemma \ref{le24}, imply that, for $n$ large enough (such that $\|\nabla |u_n|\|_2^2<\ell^2+\kappa$),
		\begin{equation*}
		\int_{\mathbb{R}^{2}}\vert u_{n}\vert^{q}(e^{\alpha\vert u_{n}\vert^{2}}-1)dx
        \leq
        \Vert u_{n}\Vert_{qr'}^{q}\Big(\int_{\mathbb{R}^{2}}(e^{r\alpha(\ell^2+\kappa)\Big(\frac{\vert u_{n}\vert}{\sqrt{\ell^2+\kappa}}\Big)^{2}}-1)dx\Big)^{1/r}dx
		\leq
		C \Vert u_{n}\Vert_{qr'}^{q}
		\end{equation*}
		and so, using \eqref{kes} and \eqref{zetaC} we can conclude.
	\end{proof}

Now we prove a fundamental result on the $(PS)_{d}$ sequences for $J_{\varepsilon}$ in the case $V_{\infty}<\infty$.

\begin{lemma}\label{compactness}
Let $d\in \mathbb{R}$, $V_{\infty}< +\infty$ and $\{u_{n}\}\subset H_{\varepsilon}$ be a $(PS)_{d}$ sequence for $J_{\varepsilon}$ such that
$\ell^2:=\limsup_{n}\| \nabla |u_n| \|_2^2  <1 $  and $u_{n}\rightharpoonup 0$ in $H_{\varepsilon}$.
If $u_{n}\not\rightarrow 0$ in $H_{\varepsilon}$, then $d\geq c_{V_{\infty}}$.
\end{lemma}
\begin{proof}
For every $n\in \mathbb{N}$, let $t_n:=t(|u_{n}|)\in (0, \infty)$ be such that $t_{n}\vert u_{n}\vert \in \mathcal{N}_{V_{\infty}}$ (see the first part of Section \ref{Secthree}).\\
Firstly we show that $\lim\sup_{n}t_{n}\leq 1$.\\
Assume by contradiction that there exist $\delta>0$ and a subsequence, still denoted by $\{t_{n}\}$, such that, for every $n\in \mathbb{N}$,
\begin{equation}\label{large}
t_{n}\geq 1+\delta.
\end{equation}
Since $\{u_{n}\}$ is a $(PS)_{d}$ sequence for $J_{\varepsilon}$, we have
\begin{equation}\label{large1}
\int_{\mathbb{R}^{2}}(\vert\nabla_{A_{\varepsilon}}u_{n}\vert^{2}+ V_\varepsilon (x)\vert u_{n}\vert^{2})dx=\int_{\mathbb{R}^{2}}f(\vert u_{n}\vert^{2})\vert
u_{n}\vert^{2}dx+o_{n}(1).
\end{equation}
On the other hand, $t_{n}\vert u_{n}\vert\in\mathcal{N}_{V_{\infty}}$, then one has
\begin{equation}\label{large2}
\int_{\mathbb{R}^{2}}\Big(\vert \nabla\vert u_{n}\vert \vert^{2}+ V_{\infty}\vert u_{n}\vert^{2}\Big)dx=\int_{\mathbb{R}^{2}}f(t_{n}^{2}\vert
u_{n}\vert^{2})\vert u_{n}\vert^{2}dx.
\end{equation}
Putting together \eqref{large1}, \eqref{large2} and using the diamagnetic inequality \eqref{2.1} we obtain
\begin{equation*}
\int_{\mathbb{R}^{2}}\Big(f(t_{n}^{2}\vert u_{n}\vert^{2})-f(\vert u_{n}\vert^{2})\Big)\vert u_{n}\vert^{2}dx\leq \int_{\mathbb{R}^{2}}(V_{\infty}-V_\varepsilon(
x))\vert u_{n}\vert^{2}dx+o_{n}(1).
\end{equation*}
Now, by the assumption \eqref{V}, we can see that for every $\eta>0$ there exists $R=R(\eta)>0$ such that, for every $\vert x\vert\geq R/\varepsilon$,
\begin{equation*}
V_{\infty}-V_\varepsilon(x)\leq \eta.
\end{equation*}
Then, using the boundedness of $\{u_{n}\}$ in $H_{\varepsilon}$ and that $u_{n}\rightarrow 0$ in $L^{2}(B_{R/\varepsilon}(0), \mathbb{C})$ by Lemma \ref{lem21},
\[
\int_{\mathbb{R}^{2}}(V_{\infty}-V(\varepsilon x))\vert u_{n}\vert^{2}dx
\leq
V_{\infty}\int_{B_{R/\varepsilon}(0)}\vert u_{n}\vert^{2}dx
+\eta\int_{B_{R}^{c}(0)}\vert u_{n}\vert^{2}dx
\leq o_{n}(1)+ C\eta.
\]
Thus,
\begin{equation}\label{large5}
\int_{\mathbb{R}^{2}}\Big(f(t_{n}^{2}\vert u_{n}\vert^{2})-f(\vert u_{n}\vert^{2})\Big)\vert u_{n}\vert^{2}dx \leq  o_{n}(1)+ C\eta.
\end{equation}
Moreover, since $u_{n}\not\rightarrow 0$ in $H_{\varepsilon}$, by Lemma \ref{le35}, there exist a sequence $\{y_{n}\}\subset \mathbb{R}^{2}$,
and two positive numbers $R$, $\beta$ such that
\begin{equation}\label{large6}
\liminf_{n\rightarrow\infty}\int_{B_{R}(y_{n})}\vert u_{n}\vert^{2}dx\geq \beta.
\end{equation}
Now, let us consider $v_{n}:=\vert u_{n}\vert(\cdot+ y_{n})$. Taking into account \eqref{V}, the diamagnetic inequality \eqref{2.1}, and the boundedness of  $\{u_{n}\}$ in
$H_{\varepsilon}$, we can see that
\begin{equation*}
\| \nabla v_{n} \|_2^2 + V_{0} \| v_{n} \|_2^2=\| \nabla |u_n| \|_2^2 + V_{0} \| u_n \|_2^2\leq \Vert u_{n}\Vert_{\varepsilon}^{2}\leq C.
\end{equation*}
Therefore $v_{n}\rightharpoonup v$ in $H^{1}(\mathbb{R}^{2}, \mathbb{R})$,  $v_{n}\rightarrow v$ in $L^{r}_{\text{loc}}(\mathbb{R}^{2}, \mathbb{R})$ for all $r\geq 1$, and, by \eqref{large6},
\begin{equation*}
\int_{B_{R}(0)}\vert v\vert^{2}dx=\lim_{n\rightarrow\infty}\int_{B_{R}(0)}\vert v_{n}\vert^{2}dx\geq\beta.
\end{equation*}
Thus there exists $\Omega\subset B_{R}(0)$ with positive measure  such that $v\not\equiv 0$ in $\Omega$.\\
Moreover, by \eqref{large}, (\ref{f4}), and \eqref{large5}, we obtain
\begin{equation*}
\int_{\mathbb{R}^{2}}\Big(f((1+\delta)^{2}v_{n}^{2})-f(v_{n}^{2})\Big)v_{n}^{2}dx\leq o_{n}(1)+ C\eta,
\end{equation*}
and so, (\ref{f4}) and the Fatou's Lemma imply
\begin{equation*}
0<\int_{\Omega}\Big(f((1+\delta)^{2}v^{2})-f(v^{2})\Big)v^{2}dx\leq  C \eta.
\end{equation*}
Thus, by the arbitrariness of $\eta>0$, we get a contradiction.\\
Now, two cases can occur.\\
{\bf Case 1:} $\limsup_{n}t_{n}=1$.\\
In this case there exists a subsequence still denoted by $\{t_{n}\}$ such that $t_{n}\rightarrow 1$.\\
Taking into account that $\{u_{n}\}$ is a $(PS)_{d}$ sequence for $J_{\varepsilon}$, \eqref{maxmu}, \eqref{cmu}, and the diamagnetic inequality \eqref{2.1}, we
have
\begin{equation}
\label{large7}
\begin{split}
d+o_{n}(1)
&=
J_{\varepsilon}(u_{n})
\geq
J_{\varepsilon}(u_{n})-I_{V_{\infty}}(t_{n}\vert u_{n}\vert )+c_{V_{\infty}}\\
&\geq
\frac{1-t_{n}^{2}}{2} \int_{\mathbb{R}^{2}}\vert \nabla\vert u_{n}\vert\vert^{2}dx
+\frac{1}{2}\int_{\mathbb{R}^{2}}\Big(V_\varepsilon(x)-t_{n}^{2}V_{\infty}\Big)\vert u_{n}\vert^{2}dx\\
&\qquad
+\frac{1}{2}\int_{\mathbb{R}^{2}}\Big(F(t_{n}^{2}\vert u_{n}\vert^{2})-F(\vert u_{n}\vert^{2})\Big)dx
+c_{V_{\infty}}.
\end{split}
\end{equation}
Since $\{\vert u_{n}\vert\}$ is bounded in $H^{1}(\mathbb{R}^{2}, \mathbb{R})$ and $t_{n}\rightarrow 1$, it is easy to see that
\begin{equation}\label{large8}
\frac{1-t_{n}^{2}}{2} \int_{\mathbb{R}^{2}}\vert \nabla\vert u_{n}\vert\vert^{2}dx=o_{n}(1).
\end{equation}
Moreover, by (\ref{V}), we have that for every $\eta>0$ there exists $R=R(\zeta)>0$ such that for any $x\in \mathbb{R}^{2}$ with $\vert  x\vert> R/\varepsilon$, it holds
\begin{equation*}
V_\varepsilon(x)-t_{n}^{2}V_{\infty}
=(V_\varepsilon(x)-V_{\infty})+(1-t_{n}^{2})V_{\infty}
\geq -\eta+(1-t_{n}^{2})V_{\infty}.
\end{equation*}
Thus, using as before the boundedness of  $\{u_{n}\}$  in $H_{\varepsilon}$, that $\vert u_{n}\vert\rightarrow 0$ in $L^{r}(B_{R/\varepsilon})$ for $r\geq 1$, and $t_{n}\rightarrow 1$, for $n$ large enough we get
\begin{equation}
\label{large9}
\begin{split}
\int_{\mathbb{R}^{2}}\Big(V_\varepsilon (x)-t_{n}^{2}V_{\infty}\Big)\vert u_{n}\vert^{2}dx
&\geq(V_{0}-t_{n}^{2}V_{\infty})\int_{B_{R/\varepsilon}(0)}\vert u_{n}\vert^{2}dx\\
&\qquad
+((1-t_{n}^{2}) V_{\infty} -\eta ) \int_{B_{R/\varepsilon}^{c}(0)}\vert u_{n}\vert^{2}dx\\
&\geq
o_{n}(1)-C\eta.
\end{split}
\end{equation}
Finally let $\vartheta,l>1$ such that $l\vartheta\ell^2<1$. Since $t_n \to 1$, for $n$ large enough,  applying the Mean Value Theorem and using also (\ref{f4}), we get
\[
\vert F(t_{n}^{2}\vert u_{n}\vert^{2})-F(\vert u_{n}\vert^{2})\vert
\leq
\vert t_{n}^{2}-1\vert f(\theta_{n}(x)\vert u_{n}\vert^{2}) \vert u_{n}\vert^{2}
\leq
\vert t_{n}^{2}-1\vert f(\vartheta \vert u_{n}\vert^{2}) \vert u_{n}\vert^{2}.
\]
Then
\begin{equation}
	\label{large10}
	\Big\vert\int_{\mathbb{R}^{2}}\Big(F(t_{n}^{2}\vert u_{n}\vert^{2})-F(\vert u_{n}\vert^{2})\Big)dx\Big\vert
	\leq
	\vert t_{n}^{2}-1\vert \int_{\mathbb{R}^{2}} f(\vartheta\vert
	u_{n}\vert^{2}) \vert u_{n}\vert^{2}dx
	=o_{n}(1).
\end{equation}
Indeed, $t_{n}\rightarrow 1$. Moreover, if we take
\[
\kappa\in\Big(0,\frac{1}{l\vartheta}-\ell^2\Big)
\text{ and }
\alpha\in\Big(4\pi,\frac{4\pi}{l\vartheta(\ell^2+\kappa)}\Big),
\]
by \eqref{2.8}, fixed $q>2$, for any $\zeta>0$, there exists a constant $C>0$ such that
\[
\int_{\mathbb{R}^{2}} f(\vartheta\vert
u_{n}\vert^{2}) \vert u_{n}\vert^{2}dx
\leq
\zeta\| u_{n}\|_2^{2}
+C\int_{\mathbb{R}^{2}}\vert u_{n}\vert^{q}(e^{\vartheta\alpha \vert	u_{n}\vert^{2}}-1)dx
\]
and the right hand side is bounded due to the boundedness of $\{u_n\}$ and since the H\"older inequality, \eqref{ineqe}, Lemma \ref{le24}, and
$$\Vert \nabla\vert u_{n}\vert\Vert_{2}^{2}
< \ell^2+\kappa<\frac{1}{l\vartheta}<1
\text{ for }n\text{ large enough},
$$
imply that
\begin{align*}
\int_{\mathbb{R}^{2}}\vert u_{n}\vert^{q}(e^{\vartheta\alpha \vert	u_{n}\vert^{2}}-1)dx
&\leq
\|u_{n}\|_{ql'}^q
\big(\int_{\mathbb{R}^{2}}(e^{l\vartheta\alpha \vert u_{n}\vert^{2}}-1)dx\big)^{1/l}\\
&\leq
\|u_{n}\|_{ql'}^q
\big(\int_{\mathbb{R}^{2}}(e^{l\vartheta\alpha(\ell^2+\kappa) \big( \frac{\vert u_{n}\vert}{\sqrt{\ell^2+\kappa}}\big)^{2}}-1)dx\big)^{1/l}
\leq
C,
\end{align*}
where $l'>1$ is the conjugate exponent of $l$.\\
Hence, putting together \eqref{large7}, \eqref{large8}, \eqref{large9}, and \eqref{large10}, we can obtain
\begin{align*}
d+C\eta+o_{n}(1)\geq c_{V_{\infty}},
\end{align*}
and, due to the arbitrariness of $\eta$, taking the limit as $n\rightarrow\infty$, we conclude.\\
{\bf Case 2:} $\lim\sup_{n}t_{n}=\bar{t}<1$.\\
In this case there exists a subsequence still denoted by $\{t_{n}\}$, such that $t_{n}\rightarrow \bar{t}$ and $t_{n}<1$ for $n$ large enough.\\
Since $\{u_{n}\}$ is a bounded $(PS)_{d}$ sequence  for $J_{\varepsilon}$ in $H_{\varepsilon}$, we have
\begin{equation}\label{large11}
d+o_{n}(1)
=J_{\varepsilon}(u_{n})-\frac{1}{2} J'_{\varepsilon}(u_{n}) [ u_{n}]
=\frac{1}{2}\int_{\mathbb{R}^{2}}\Big(f(\vert u_{n}\vert^{2})\vert
u_{n}\vert^{2}-F(\vert u_{n}\vert^{2})\Big)dx.
\end{equation}
Note that, by (\ref{f3}), the map $t\mapsto f(t)t-F(t)$ is increasing for $t>0$. Hence, this fact combined with $t_{n}\vert u_{n}\vert\in\mathcal{N}_{V_{\infty}}$, $t_{n}<1$, and \eqref{large11}, allows us to get
\begin{align*}
c_{V_{\infty}}&\leq I_{V_{\infty}}(t_{n}\vert u_{n}\vert)
=I_{V_{\infty}}(t_{n}\vert u_{n}\vert)
-\frac{t_{n}}{2} I'_{V_{\infty}}(t_{n}\vert u_{n}\vert) [\vert u_{n}\vert]\\
&=\frac{1}{2}\int_{\mathbb{R}^{2}}\Big(f(t_{n}^{2}\vert u_{n}\vert^{2})t_{n}^{2}\vert u_{n}\vert^{2}-F(t_{n}^{2}\vert u_{n}\vert^{2})\Big)dx\\
&\leq\frac{1}{2}\int_{\mathbb{R}^{2}}\Big(f(\vert u_{n}\vert^{2})\vert u_{n}\vert^{2}-F(\vert u_{n}\vert^{2})\Big)dx\\
&= d+o_{n}(1).
\end{align*}
Passing to the limit as $n\rightarrow +\infty$, we complete the proof.
\end{proof}
We shall also use the next result about the $(PS)$ sequences.
\begin{lemma}\label{gt6888}
Let $\{u_{n}\}\subset H_\varepsilon$ be a (PS)  sequence for $J_{\varepsilon}$ such that   $\ell^2:=\limsup_{n}\Vert u_{n}\Vert_{\varepsilon}^{2}<1$.
Then, up to a subsequence, $\nabla\vert u_{n}\vert\rightarrow \nabla\vert u\vert$ a.e. in $\mathbb{R}^{2}$.
\end{lemma}
\begin{proof}
Let $\{u_{n}\}\subset H_\varepsilon$ be a (PS)  sequence for $J_{\varepsilon}$ with $\limsup_{n}\Vert u_{n}\Vert_{\varepsilon}^{2}<1$. By Lemma \ref{PSbdd} we have that $\{u_{n}\}$ is bounded in $H_{\varepsilon}$ and so, up to a subsequence, $u_{n}\rightharpoonup u$ in $H_{\varepsilon}$, $u_{n}\rightarrow u$ in $L^{r}_{\rm loc}(\mathbb{R}^{2}, \mathbb{C})$ for any $r\geq 1$ and $u_{n}\rightarrow u$ a.e. in $\mathbb{R}^{2}$. Since the norm is weakly lower semicontinuous, it follows that $\Vert u\Vert_{\varepsilon}^{2}<1$.\\
Now observe that, for every $\phi\in C_{c}^{\infty}(\mathbb{R}^{2},\mathbb{C})$,
\[
f(\vert u_{n}\vert^{2})u_{n}\overline{\phi } \rightarrow  f(\vert u\vert^{2})u\overline{\phi}\,
\text{ a.e. in } \,\mathbb{R}^2\,
\text{ as }\, n\to+\infty
\]
and, by \eqref{2.8}, fixed $q>2$, for any $\zeta>0$ and $\alpha>4\pi$, there exists $C>0$ such that, for every $\phi\in C_{c}^{\infty}(\mathbb{R}^{2},\mathbb{R})$,
\[
|{\rm Re}[f(|u_{n}^{2}|)u_{n}\overline{\phi}] |
\leq
|f(|u_{n}^{2}|)u_{n}\overline{\phi} |
\leq
\zeta | u_{n}| | \phi|
+C | u_{n}|^{q-1} (e^{\alpha |u_{n}|^{2}}-1)| \phi|
\]
with
\[
\zeta | u_{n}| | \phi|
+C | u_{n}|^{q-1}(e^{\alpha |u_{n}|^{2}}-1)| \phi|
\to
\zeta | u| | \phi|
+C | u|^{q-1} (e^{\alpha |u|^{2}}-1)| \phi|
\text{ a.e. in } \mathbb{R}^2 \text{ as }n\to+\infty.
\]
Moreover, let $r^*,\tau>1$ such that $r^*\tau\ell^2<1$. Using the H\"older inequality, \eqref{ineqe}, Sobolev inequality, Lemma \ref{le24}, and taking
\[
\kappa\in\Big(0,\frac{1}{r^*\tau}-\ell^2\Big)
\text{ and }
\alpha\in\Big(4\pi,\frac{4\pi}{r^*\tau (\ell^2+\kappa)}\Big),
\]
we have that, for $n$ large enough, such that $\|\nabla|u_n|\|_2^2\leq \ell^2+\kappa$,
\begin{equation}
	\label{bddun}
	\begin{split}
		\int_{\mathbb{R}^{2}}\Big[\vert u_{n}\vert^{q-1}(e^{\alpha |u_{n}|^{2}}-1)\Big]^{r^*} dx
		&\leq
		\|u_n\|_{\tau' r^*(q-1)}^{r^*(q-1)}
		\Big( \int_{\mathbb{R}^{2}} (e^{\alpha r^* \tau |u_{n}|^{2}}-1) dx\Big)^{1/\tau}\\
		&\leq
			C
			\|u_n\|_{\varepsilon}^{r^*(q-1)}
			\Big( \int_{\mathbb{R}^{2}} (e^{\alpha r^* \tau (\ell^2+\kappa) \Big(\frac{\vert u_{n}\vert}{\sqrt{\ell^2+\kappa}}\Big)^{2}}-1) dx\Big)^{1/\tau}
			\leq
			C,
	\end{split}
\end{equation}
where $\tau'$ is the conjugate exponent of $\tau$.\\
Then, up to a subsequence and applying the Severini-Egoroff Theorem, we get
\[
\vert u_{n}\vert^{q-1}(e^{\alpha |u_{n}|^{2}}-1)
\rightharpoonup
\vert u\vert^{q-1}(e^{\alpha |u|^{2}}-1)
\text{ in } L^{r}(\mathbb{R}^{2}, \mathbb{R}).
\]
Thus, since, up to a subsequence, also $\vert u_{n}\vert \rightharpoonup \vert u\vert$ in $L^{2}(\mathbb{R}^{2}, \mathbb{R})$, we have, as $n\to +\infty$,
\[
\zeta\int_{\mathbb{R}^{2}}\vert u_{n}\vert \vert \phi\vert dx
+C\int_{\mathbb{R}^{2}}\vert u_{n}\vert^{q-1}(e^{\alpha |u_{n}|^{2}}-1)\vert \phi\vert dx
\to
\zeta\int_{\mathbb{R}^{2}}\vert u \vert \vert \phi\vert dx
+C\int_{\mathbb{R}^{2}}\vert u \vert^{q-1}(e^{\alpha |u|^{2}}-1)\vert \phi\vert dx.
\]
Hence, by a variant of the Lebesgue Dominated Convergence Theorem,
we obtain that, for every $\phi\in C_{c}^{\infty}(\mathbb{R}^{2},\mathbb{C})$
$$
\text{Re}\int_{\mathbb{R}^{2}}f(\vert u_{n}\vert^{2})u_{n}\overline{\phi }dx\rightarrow \text{Re}\int_{\mathbb{R}^{2}}f(\vert u\vert^{2})u\overline{\phi} dx,
$$
from which we get $J'_{\varepsilon}(u)=0$.\\
Let now $R>0$ be arbitrary and $\psi$ be a $C_c^{\infty}(\mathbb{R}^2,[0,1])$ function such that
$\psi\equiv 1$ in $B_{R}(0)$ and $\psi\equiv 0$ in $B_{2R}^c(0)$.\\
First we observe that, since for every $v\in H_\varepsilon$,
\[
\nabla_{A_\varepsilon}(v\psi)=\psi\nabla_{A_\varepsilon}v-iv\nabla \psi,
\]
and then, by the H\"older inequality,
\[
\|\nabla_{A_\varepsilon}(v\psi)\|_2^2
=
\|\psi\nabla_{A_\varepsilon}v\|_2^2
+2\text{Re}\int_{\mathbb{R}^2}i\overline{v}\psi \nabla_{A_\varepsilon}v\nabla \psi dx
+\||v| |\nabla\psi|\|_2^2
\leq
C(\|\nabla_{A_\varepsilon}v\|_2^2 + \|v\|_2^2),
\]
we have that $v\psi \in H_\varepsilon$ and that the boundedness of $\{u_n\}$ in $H_\varepsilon$ implies the boundedness of $\{u_n\psi\}$ in $H_\varepsilon$.\\
By \eqref{2.8}, the H\"{o}lder and diamagnetic inequalities, \eqref{ineqe}, Lemma \ref{lem21}, and Lemma \ref{le24}, for  $q>2$ and $\alpha>4\pi$,
\begin{equation*}
\begin{split}
\left|\text{Re}\int_{\mathbb{R}^{2}}f(\vert u\vert^{2})u\overline{ (u-u_{n})}\psi dx\right|
&\leq
\int_{\mathbb{R}^{2}}f(\vert u\vert^{2})|u| |u-u_{n}|\psi dx\\
&\leq
\zeta \int_{\mathbb{R}^{2}}|u| |u-u_{n}|\psi dx\\
&\qquad
+ C \int_{\mathbb{R}^{2}} |u|^{q-1}  (e^{\alpha |u|^2}-1) |u-u_{n}|\psi dx\\
&\leq
\zeta \|u\|_2 \left(\int_{B_{2R}(0)} |u-u_{n}|^2 dx\right)^{\frac{1}{2}}\\
&\qquad
+ C \left(\int_{\mathbb{R}^{2}} |u|^{2(q-1)}  (e^{2\alpha |u|^2}-1) dx\right)^{\frac{1}{2}}
\left(\int_{B_{2R}(0)} |u-u_{n}|^{2} dx\right)^{\frac{1}{2}}\\
&\leq
o_n(1)
+ C \|u\|_{4(q-1)}^{q-1}
\left(\int_{\mathbb{R}^{2}}(e^{4\alpha |u|^2}-1) dx\right)^{\frac{1}{4}}
o_n(1)\\
&= o_n(1).
\end{split}
\end{equation*}
Moreover, by Lemma \ref{lem21} and the H\"older inequality,
\[
\left|\text{Re}\int_{\mathbb{R}^{2}}V_\varepsilon(x)u\overline{ (u-u_{n})}\psi dx\right|
\leq
C \| u\|_2 \left(\int_{B_{2R}(0)} |u-u_{n}|^{2} dx\right)^{\frac{1}{2}}
=o_n(1)
\]
and
\[
\left|\text{Re}\int_{\mathbb{R}^{2}}i \overline{ (u-u_{n})}\nabla_{A_\varepsilon} u \nabla \psi  dx\right|
\leq
C \| \nabla_{A_\varepsilon} u\|_2 \left(\int_{B_{2R}(0)\setminus B_R(0)} |u-u_{n}|^{2} dx\right)^{\frac{1}{2}}
=o_n(1).
\]
Thus
\begin{equation}
\label{more2}
\begin{split}
\text{Re}\int_{\mathbb{R}^{2}} \nabla_{A_\varepsilon} u\overline{\nabla_{A_\varepsilon} (u-u_{n})}\psi dx
&=
J'_{\varepsilon}(u) [ (u-u_{n})\psi]
-\text{Re}\int_{\mathbb{R}^{2}}V_\varepsilon(x)u\overline{ (u-u_{n})}\psi dx\\
&\quad
-\text{Re}\int_{\mathbb{R}^{2}}i \overline{ (u-u_{n})}\nabla_{A_\varepsilon} u \nabla \psi  dx
+\text{Re}\int_{\mathbb{R}^{2}}f(\vert u\vert^{2})u\overline{ (u-u_{n})}\psi dx\\
&=o_{n}(1).
\end{split}
\end{equation}
Now, using \eqref{more2}, $J'_{\varepsilon}(u_{n})\rightarrow 0$ as $n\rightarrow +\infty$, and the boundedness of the sequence $\{u_{n}\psi\}$ in $H_\varepsilon$,  we obtain
\begin{equation}\label{weak101}
		\begin{split}
			0\leq
			\int_{B_{R}(0)} | \nabla_{A_\varepsilon} u_{n}-  \nabla_{A_\varepsilon} u|^2 dx
			&\leq
			\int_{\mathbb{R}^{2}} | \nabla_{A_\varepsilon} u_{n}-  \nabla_{A_\varepsilon} u|^2\psi dx\\
			&=
			\text{Re}\int_{\mathbb{R}^{2}} \nabla_{A_\varepsilon} u_{n}\overline{\nabla_{A_\varepsilon} (u_n-u)}\psi dx
			+o_{n}(1)\\
			&= J'_{\varepsilon}(u_{n}) [ (u_{n}-u)\psi]
			+K_{1, n}+K_{2, n}+K_{3, n}+o_{n}(1)\\
			&= K_{1, n}+K_{2, n}+K_{3, n}+o_{n}(1),
		\end{split}
	\end{equation}
where
\begin{align*}
K_{1, n}
&=
\text{Re}\int_{\mathbb{R}^{2}}i \overline{( u-u_n)}\nabla_{A_\varepsilon} u_{n}\nabla \psi dx,\\
K_{2, n}
&=
\text{Re}\int_{\mathbb{R}^{2}} V_\varepsilon (x) u_{n}\overline{( u-u_n)} \psi dx,\\
K_{3, n}
&=
\text{Re}\int_{\mathbb{R}^{2}}f(\vert u_{n}\vert^{2})u_{n}\overline{( u_n-u)} \psi dx.
\end{align*}
By Lemma \ref{lem21}, the H\"{o}lder inequality, and the boundedness of $\{u_n\}$ in $H_\varepsilon$, we get
\begin{equation*}
\vert K_{1, n}\vert
\leq
\int_{B_{2R}(0)\backslash B_{R}(0)} \vert \nabla_{A_\varepsilon} u_{n}\vert \vert  u-u_{n}\vert \vert \nabla \psi \vert dx
\leq C \Vert u_{n}\Vert_{\varepsilon} \Big(\int_{B_{2R}(0)\backslash B_{R}(0)} \vert  u-u_{n}\vert^{2}dx\Big)^{1/2}=o_{n}(1)
\end{equation*}
and
\begin{equation*}
|K_{2, n}|
\leq
\int_{B_{2R}(0)} V_\varepsilon (x) |u_{n}| |u-u_n| dx
\leq
C \| u_n\|_2 \left(\int_{B_{2R}(0)} |u-u_{n}|^{2} dx\right)^{1/2}
=o_n(1).
\end{equation*}
Moreover, by \eqref{2.8}, the H\"{o}lder inequality and the boundedness of $\{u_n\}$ in $H_\varepsilon$, for $q>2$ and suitable $\tau>1$ and $\alpha>4\pi$, if $\tau'$ is the conjugate exponent of $\tau$,
	\begin{equation*}
		\begin{split}
			|K_{3, n}|
			&\leq
			\int_{B_{2R}(0)}|f(\vert u_{n}\vert^{2})| |u_{n}| |u_{n}-u| dx\\
			&\leq
			\zeta \| u_{n}\|_2 \left(\int_{B_{2R}(0)} \vert u_{n}-u\vert^2 dx\right)^{1/2}
			+C \int_{B_{2R}(0)}\vert u_{n}-u\vert \vert u_{n}\vert^{q-1}(e^{\alpha\vert u_{n}\vert^{2} }-1) dx\\
			&\leq
			o_n(1)
			+C\left(\int_{B_{2R}(0)} \vert u_{n}-u\vert^{\tau'} dx\right)^{1/\tau'}
			\left(\int_{\mathbb{R}^{2}}\Big[\vert u_{n}\vert^{q-1}(e^{\alpha |u_{n}|^{2}}-1)\Big]^{\tau} dx\right)^{1/\tau}
			\\
			&=
			o_n(1),
		\end{split}
	\end{equation*}
since, arguing as in \eqref{bddun}, for $n$ large enough,
\[
\int_{\mathbb{R}^{2}}\Big[\vert u_{n}\vert^{q-1}(e^{\alpha |u_{n}|^{2}}-1)\Big]^{\tau} dx
\leq C.
\]
Hence, by \eqref{weak101}, we obtain that, for any $R>0$,
$$
\int_{B_{R}(0)} | \nabla_{A_\varepsilon} u_{n}-  \nabla_{A_\varepsilon} u|^2 dx\rightarrow 0,\,\,\text{as}\,\,n\rightarrow +\infty,
$$
which implies that $\nabla_{A_\varepsilon} u_{n} \to  \nabla_{A_\varepsilon} u$ a.e. in $\mathbb{R}^2$. Thus, using also that $u_n \to u$ a.e. in $\mathbb{R}^2$, we have
\begin{equation*}
\nabla u_{n} \to  \nabla u \text{ a.e. in } \mathbb{R}^2
\end{equation*}
and we can conclude.
\end{proof}

Now we provide a range of levels in which the functional $J_{\varepsilon}$ satisfies the $(PS)$ condition.
\begin{proposition}\label{gt688}
Let $c\in \mathbb{R}$ and $\{u_{n}\}$ be a $(PS)_{c}$ sequence for $J_{\varepsilon}$ such that $\lim\sup_{n\rightarrow\infty}\Vert u_{n}\Vert_{\varepsilon}^2<1/8$.
Assume that $c< c_{V_{\infty}}$ if $V_{\infty}<+\infty$ or $c$ is arbitrary if $V_{\infty}=+\infty$. Then, $\{u_{n}\}$ admits a convergent subsequence in $H_{\varepsilon}$.
\end{proposition}

\begin{proof}
Since the sequence $\{u_{n}\}$ is bounded in $H_{\varepsilon}$, up to a subsequence, $u_{n}\rightharpoonup u$ in $H_{\varepsilon}$ and $u_{n}\rightarrow u$ in $L^{r}_{\rm loc}(\mathbb{R}^{2}, \mathbb{C})$ for any $r\geq 1$. Moreover, arguing as in Lemma \ref{gt6888}, we also have  that, for any $\phi\in H_{\varepsilon}$,
\begin{equation*}
\text{Re}\int_{\mathbb{R}^{2}}f(\vert u_{n}\vert^{2})u_{n}\overline{\phi}dx\rightarrow \text{Re}\int_{\mathbb{R}^{2}}f(\vert u\vert^{2})u\overline{\phi}dx
\text{ as } n\rightarrow +\infty.
\end{equation*}
Thus $u$ is a critical point of $J_{\varepsilon}$ and, using (\ref{f3}),
\begin{equation}\label{weak3}
J_{\varepsilon}(u)
=
J_{\varepsilon}(u)-\frac{1}{\theta} J'_{\varepsilon}(u) [u]
= \frac{\theta-2}{2\theta} \|u\|_\varepsilon^2
+ \frac{1}{\theta}\int_{\mathbb{R}^{2}}\Big(f(\vert u\vert^{2})\vert u\vert^{2}-\frac{\theta}{2}F(\vert u\vert^{2})\Big)dx\geq 0.
\end{equation}
Now, let $\{u_{n_j}\}$ be a subsequence of $\{u_n\}$ satisfying the assumptions of Lemma \ref{gt6888} and Lemma \ref{2.0} and let  $v_{j}:=u_{n_{j}}-\hat{u}_{j}$ where $\hat{u}_{j}=\varphi_j u$, with $\varphi_{j}(x)=\varphi(2x/j)$, $\varphi\in C_{c}^{\infty}(\mathbb{R}^{2}, \mathbb{R})$, $0\leq \varphi\leq 1$,
$\varphi(x)=1$ if $\vert x\vert\leq 1$, and $\varphi(x)=0$ if $\vert x\vert\geq 2$. We claim that
\begin{equation}\label{weak4}
J_{\varepsilon}(v_{j})=c-J_{\varepsilon}(u)+o_{j}(1),
\end{equation}
and
\begin{equation}\label{weak5}
J'_{\varepsilon}(v_{j})=o_{j}(1).
\end{equation}
In order to prove \eqref{weak4}, let us observe that
\[
J_{\varepsilon}(v_{j})-J_{\varepsilon}(u_{n_{j}})+J_{\varepsilon}(\hat{u}_{j})
=
\Vert \hat{u}_{j}\Vert_{\varepsilon}^{2}-\langle u_{n_{j}}, \hat{u}_{j}\rangle_{\varepsilon}
+\frac{1}{2}\int_{\mathbb{R}^{2}}( F(\vert u_{n_{j}}\vert^{2})-F(\vert v_{j}\vert^{2})-F(\vert \hat{u}_{j}\vert^{2})) dx.
\]
In view of the weak convergence of $\{u_{n_{j}}\}$ to $u$ in $H_{\varepsilon}$  and using Lemma \ref{gt521}, we can see that
\[
\Vert \hat{u}_{j}\Vert_{\varepsilon}^{2}-\langle u_{n_{j}}, \hat{u}_{j}\rangle_{\varepsilon}
= o_j(1).
\]
Moreover, by Lemma \ref{gt6888} and \eqref{apprF} in Lemma \ref{2.0}, we have that
\[
\int_{\mathbb{R}^{2}}( F(\vert u_{n_{j}}\vert^{2})-F(\vert v_{j}\vert^{2})-F(\vert \hat{u}_{j}\vert^{2})) dx
= o_j(1),
\]
getting \eqref{weak4}.\\
To prove \eqref{weak5}, we observe that, for any $\phi\in H_{\varepsilon}$,
\begin{align*}
\vert  (J'_{\varepsilon}(v_{j}) -J'_{\varepsilon}(u_{n_{j}}) +J'_{\varepsilon}(\hat{u}_{j}))[\phi] \vert&=\Big\vert\text{Re}\int_{\mathbb{R}^{2}}(f(\vert
u_{n_{j}}\vert^{2})u_{n_{j}}-f(\vert v_{j}\vert^{2})v_{j}-f(\vert \hat{u}_{j}\vert^{2})\hat{u}_{j})\overline{\phi} dx\Big\vert \\
&\leq \int_{\mathbb{R}^{2}}\vert f(\vert u_{n_{j}}\vert^{2})u_{n_{j}}-f(\vert v_{j}\vert^{2})v_{j}-f(\vert \hat{u}_{j}\vert^{2})\hat{u}_{j}\vert \vert \phi\vert
dx
\end{align*}
and so, by Lemma \ref{gt6888} and \eqref{apprf} in Lemma \ref{2.0}, applying the diamagnetic inequality \eqref{2.1}, we have that
\[
\sup_{\Vert \phi\Vert_{\varepsilon}\leq 1}
\vert ( J'_{\varepsilon}(v_{j})-J'_{\varepsilon}(u_{n_{j}})+J'_{\varepsilon}(\hat{u}_{j})) [\phi] \vert\rightarrow 0.
\]
Thus, since $J'_{\varepsilon}(u_{n_{j}})\rightarrow 0$
and $J'_{\varepsilon}(\hat{u}_{j})\rightarrow J'_{\varepsilon}(u)=0$, \eqref{weak5} holds.\\
Now, if $V_{\infty}<+\infty$ and $c<c_{V_{\infty}}$, by \eqref{weak3}, we have that
\begin{equation}
\label{levels}
c-J_{\varepsilon}(u)\leq c<c_{V_{\infty}}.
\end{equation}
Moreover, by \eqref{weak4} and \eqref{weak5}, $\{v_{j}\}$ is a $(PS)_{c-J_{\varepsilon}(u)}$ sequence for $J_{\varepsilon}$ with $v_{j}\rightharpoonup 0$ in $H_{\varepsilon}$.\\
Note also that, due to the weakly lower semicontinuity of the norm,
\[
\| u\|_\varepsilon^2 \leq \liminf_j \| u_{n_j} \|_\varepsilon^2.
\]
Then, using also the diamagnetic inequality and Lemma \ref{gt521},
\begin{equation}\label{newineqality}
\limsup_{j}(\| \nabla |v_j| \|_2^2 + V_0 \| v_j\|_2^2)
\leq
\limsup_{j}\Vert u_{n_{j}}-\hat{u}_{j}\Vert_{\varepsilon}^{2}\\
\leq
2 (\limsup_{j}\Vert u_{n_{j}}\Vert_{\varepsilon}^{2}+\limsup_{j}\Vert\hat{u}_{j}\Vert_{\varepsilon}^{2})
<\frac{1}{2}.
\end{equation}
Hence, \eqref{levels} and Lemma \ref{compactness} imply that $v_{j}\rightarrow 0$ in $H_{\varepsilon}$ and so, by Lemma \ref{gt521}, we get that $u_{n_{j}}\rightarrow u$ in $H_{\varepsilon}$ as $j\rightarrow +\infty$.\\
If $V_{\infty}=+\infty$, then, by Lemma \ref{gt521} and Lemma \ref{lem21}, $v_{j}\rightarrow 0$ in $L^{r}(\mathbb{R}^{2}, \mathbb{C})$ for any $r\geq 2$. Arguing as in \eqref{newineqality}, we also have
$$\limsup_{j}(\| \nabla |v_j| \|_2^2 + V_0 \| v_j\|_2^2)< 1/2<1$$
and so, by \eqref{2.10}, the H\"older inequality, \eqref{ineqe}, and Lemma \ref{le24},  we deduce that, fixed $q>2$, for suitable $\tau>1$ and $\alpha>4\pi$ such that $\alpha \tau < 8\pi$, for $j$ large enough,
\begin{equation}
\label{vj}
\begin{split}
\Vert v_{j}\Vert_{\varepsilon}^{2}
&=\int_{\mathbb{R}^{2}}f(\vert v_{j}\vert^{2})\vert v_{j}\vert^{2} dx+o_{j}(1)
\leq C \int_{\mathbb{R}^{2}} |v_j|^q (e^{\alpha |v_j|^2}-1) dx +o_{j}(1)\\
&\leq C \|v_j\|_{q\tau'}^q \Big(\int_{\mathbb{R}^{2}} (e^{\alpha |v_j|^2}-1)^{\tau} dx\Big)^{1/\tau} +o_{j}(1)\\
&\leq
	C \|v_j\|_{q\tau'}^q \Big(\int_{\mathbb{R}^{2}} \Big(e^{\frac{\alpha \tau}{2} (\sqrt{2}|v_j|)^2}-1\Big) dx\Big)^{1/\tau} +o_{j}(1)
	=o_{j}(1),
\end{split}
\end{equation}
where $\tau'$ is the conjugate exponent of $\tau$.\\
Hence $u_{n_{j}}\rightarrow u$ in $H_{\varepsilon}$ as $j\rightarrow +\infty$ and the proof is completed.
\end{proof}

To show that  $\mathcal{N}_{\varepsilon}$ is a natural constraint, namely that the constrained critical points of the functional $J_{\varepsilon}$ on
$\mathcal{N}_{\varepsilon}$ are critical points of $J_{\varepsilon}$ in  $H_{\varepsilon}$, we prove the following preliminary result.

\begin{proposition}\label{gt621}
Let $c\in\mathbb{R}$ and $\{u_{n}\}$ be a $(PS)_{c}$ sequence for $J_{\varepsilon}$ restricted to $\mathcal{N}_{\varepsilon}$ with $\limsup_{n\rightarrow\infty}\Vert u_{n}\Vert_{\varepsilon}^2<1/8$. Assume that $c< c_{V_{\infty}}$ if $V_{\infty}<+\infty$ or $c$ is arbitrary if $V_{\infty}=+\infty$. Then, $\{u_{n}\}$ has a convergent subsequence in  $\mathcal{N}_{\varepsilon}$.
\end{proposition}

\begin{proof}
Let $\{u_{n}\}\subset \mathcal{N}_{\varepsilon}$ be a $(PS)_{c}$ sequence of $J_{\varepsilon}$ restricted to $\mathcal{N}_{\varepsilon}$. Then,
$J_{\varepsilon}(u_{n})\rightarrow c$ as $n\rightarrow +\infty$ and there exists $\{\lambda_{n}\}\subset \mathbb{R}$ such that
\begin{align}\label{3.15}
J'_{\varepsilon}(u_{n})=\lambda_{n}T'_{\varepsilon}(u_{n})+o_{n}(1),
\end{align}
with
\[
T_{\varepsilon}:H_{\varepsilon}\rightarrow  \mathbb{R},
\qquad
T_{\varepsilon}(u):=\Vert u\Vert_{\varepsilon}^{2}-\int_{\mathbb{R}^{2}}f(\vert u\vert^{2})\vert u\vert^{2} dx.
\]
Using (\ref{f4}), we get
\begin{align*}
T'_{\varepsilon}(u_{n}) [u_{n}]
&=2\Vert u_{n}\Vert_{\varepsilon}^{2}-2\int_{\mathbb{R}^{2}}f(\vert u_{n}\vert^{2})\vert u_{n}\vert^{2}
dx-2\int_{\mathbb{R}^{2}}f'(\vert u_{n}\vert^{2})\vert u_{n}\vert^{4} dx\\
&=-2\int_{\mathbb{R}^{2}}f'(\vert u_{n}\vert^{2})\vert u_{n}\vert^{4} dx\leq -(p-2)C_{p}\int_{\mathbb{R}^{2}}\vert u_{n}\vert^{p} dx\leq 0.
\end{align*}
Moreover $\{T'_{\varepsilon}(u_{n})\}$ is bounded. Indeed, using the H\"older inequality, \eqref{2.8}, and (\ref{f5}),
\begin{align*}
	\vert T'_{\varepsilon}(u_{n}) [\phi] \vert
	&\leq
	2\Vert u_{n}\Vert_{\varepsilon}\Vert \phi\Vert_{\varepsilon}
	+2\int_{\mathbb{R}^{2}}f(\vert u_{n}\vert^{2})\vert u_{n}\vert\vert \phi\vert dx
	+2\int_{\mathbb{R}^{2}}f'(\vert u_{n}\vert^{2}) \vert u_{n}\vert^{3} \vert \phi\vert dx\\
	&\leq C
	\left[
	\Vert u_{n}\Vert_{\varepsilon}\Vert \phi\Vert_{\varepsilon}
	+\int_{\mathbb{R}^{2}} |u_n|^{q-1} |\phi| (e^{\alpha |u_n|^2}-1) dx
	+\int_{\mathbb{R}^{2}} \vert u_{n}\vert^{3} \vert \phi\vert (e^{4\pi |u_n|^2}-1)   dx
	\right].
\end{align*}
Then, using similar arguments as in \eqref{vj} and taking $\alpha>4\pi$ and $r>1$ with $\alpha r <32 \pi$, for $n$ large enough,
\begin{align*}
\int_{\mathbb{R}^{2}} |u_n|^{q-1}  |\phi| (e^{\alpha |u_n|^2}-1)dx
&\leq
\Big( \int_{\mathbb{R}^{2}} |u_n|^{r'(q-1)}  |\phi|^{r'} dx \Big)^{1/r'}
\Big( 	\int_{\mathbb{R}^{2}}(e^{\frac{\alpha r}{8} (2\sqrt{2}\vert u_{n}\vert)^{2}}-1)dx \Big)^{1/r}\\
&\leq
C\| u_n \|_{\varepsilon}^{q-1} \|\phi\|_\varepsilon,
\end{align*}
where $r'$ is the conjugate exponent of $r$.\\
Analogously
\[
\int_{\mathbb{R}^{2}} \vert u_{n}\vert^{3} \vert \phi\vert   (e^{4\pi |u_n|^2}-1) dx
\leq
C\Vert u_{n}\Vert_{\varepsilon}^{3}\Vert \phi\Vert_{\varepsilon}.
\]
Thus, using the boundedness of $\{u_{n}\}$,  for every $\phi\in H_{\varepsilon}$ we have that
\[
\vert T'_{\varepsilon}(u_{n}) [\phi] \vert
\leq
C(\Vert u_{n}\Vert_{\varepsilon}+\Vert u_{n}\Vert_{\varepsilon}^{q-1}+\Vert u_{n}\Vert_{\varepsilon}^{3})\Vert \phi\Vert_{\varepsilon}
\leq C \Vert \phi\Vert_{\varepsilon}.
\]
Thus, up to a subsequence, $ T'_{\varepsilon}(u_{n})[ u_{n}]\rightarrow\varsigma\leq 0$. \\
We claim that $ T'_{\varepsilon}(u_{n})[ u_{n}]\not\rightarrow 0$.\\
Indeed, if  $ T'_{\varepsilon}(u_{n})[ u_{n}]\rightarrow 0$, then also $u_{n}\rightarrow 0$ in $L^{p}(\mathbb{R}^{2}, \mathbb{C})$, since
\[
o_{n}(1)=\vert T'_{\varepsilon}(u_{n}) [u_{n}] \vert\geq C \int_{\mathbb{R}^{2}}\vert u_{n}\vert^{p} dx.
\]
Thus, by interpolation, $u_{n}\rightarrow 0$ in $L^{\tau}(\mathbb{R}^{2}, \mathbb{C})$, for all $\tau > 2$.\\
Moreover, using similar arguments as before and taking $\alpha>4\pi$ and $r>1$ with $\alpha r <32 \pi$, we get that, for every $\zeta >0$ and for $n$ large enough,
	\[
	\Vert u_{n}\Vert_{\varepsilon}^{2}
	=\int_{\mathbb{R}^{2}}f(\vert u_{n}\vert^{2})\vert u_{n}\vert^{2} dx
	\leq
	\frac{\zeta}{V_{0}}\Vert u_{n}\Vert_{\varepsilon}^{2}
	+C \| u_n \|_{qr'}^q
	\Big( 	\int_{\mathbb{R}^{2}}(e^{\frac{\alpha r}{8} (2\sqrt{2}\vert u_{n}\vert)^{2}}-1)dx \Big)^{1/r'}
	=
	\frac{\zeta}{V_{0}}\Vert u_{n}\Vert_{\varepsilon}^{2}
	+ o_n(1)
	\]
and so $u_{n}\rightarrow 0$ in $H_{\varepsilon}$, which is in contradiction with Lemma \ref{unifbound}.\\
Hence, since $\{u_{n}\}\subset \mathcal{N}_{\varepsilon}$, by \eqref{3.15} we get $\lambda_{n}=o_{n}(1)$ and so $J'_{\varepsilon}(u_{n})=o_{n}(1)$. Then, applying Proposition \ref{gt688}, we conclude.
\end{proof}

 Arguing as in the first part of the proof of Proposition \ref{gt621}, we have
\begin{corollary}\label{gt678}
The constrained critical points of the functional $J_{\varepsilon}$ on $\mathcal{N}_{\varepsilon}$ are critical points of $J_{\varepsilon}$ in $H_{\varepsilon}$.
\end{corollary}

Now we can complete the proof of Theorem \ref{gt62}.
\begin{proof} [Proof of Theorem \ref{gt62}]
By Lemma \ref{mountain}, we know that the functional $J_{\varepsilon}$ has the Mountain Pass geometry. Thus,  there exists
a $(PS)_{c_{\varepsilon}}$ sequence $\{u_{n}\}\subset H_{\varepsilon}$ for $J_{\varepsilon}$, with $c_\varepsilon$ defined in \eqref{ceps}. By Lemma \ref{PSbdd} we have that $\{u_{n}\}$ is bounded in $H_{\varepsilon}$
and, by (\ref{f3}), Lemma \ref{newadd}, and \eqref{moncmu}, for $\varepsilon$ small enough,
	\begin{align*}
		\frac{\theta-2}{2\theta}\|u_n\|_{\varepsilon}^2
		&\leq
		\Big(\frac{1}{2}-\frac{1}{\theta}\Big) \|u_n\|_{\varepsilon}^2
		+\frac{1}{\theta}\int_{\mathbb{R}^{2}}\Big(f( |u_{n}|^{2}) |u_{n}|^{2}-\frac{\theta}{2}F( |u_{n}|^{2})\Big)dx
		=
		J_{\varepsilon}(u_{n})-\frac{1}{\theta}J_{\varepsilon}'(u_{n})[u_{n}]\\
		&\leq
		c_\varepsilon + o_n(1) + o_n(1) \|u_n\|_{\varepsilon}
		\leq
		c_{V_0}+ o_n(1) + o_n(1) \|u_n\|_{\varepsilon}
		\leq
		c_{V^*}+ o_n(1) + o_n(1) \|u_n\|_{\varepsilon},
	\end{align*}
Thus, using Lemma \ref{le32}, for $n$ large enough we have
\[
\frac{\theta-2}{2\theta}\|u_n\|_{\varepsilon}^2
<\frac{\theta-2}{16\theta}
\]
and so $\limsup_{n}\Vert u_{n}\Vert_{\varepsilon}^2< 1/8$.
Hence, \eqref{moncmu} and Proposition \ref{gt688} allow us to conclude.
\end{proof}

 We conclude this section showing the behavior of the energy levels $c_\varepsilon$ of the solutions found in Theorem \ref{gt62}.
\begin{proposition}\label{propcecv0}
Let $\{\varepsilon_n\}\subset (0,+\infty)$ be such that $\varepsilon_n \to 0^+$ as $n\to +\infty$ and, for $n$ large enough, $u_{n}$ be a respective ground state found in Theorem \ref{gt62}  for  $\varepsilon=\varepsilon_{n}$ with energy $J_{\varepsilon_{n}}(u_{n})=c_{\varepsilon_{n}}$. Up to a subsequence, we have that
\[
c_{\varepsilon_{n}}\rightarrow c_{V_{0}}.
\]
\end{proposition}
\begin{proof}
By Lemma \ref{unifbound} and Lemma \ref{newadd}, $0<K_2\leq c_{\varepsilon_{n}}\leq c_{V_{0}}$. Thus, up to a subsequence, $c_{\varepsilon_{n}}\rightarrow \bar{c}>0$.\\
Since, using (\ref{f3}), Lemma \ref{newadd}, the monotonicity of $c_\mu$'s, and Lemma \ref{le32},
\begin{align*}
	\frac{\theta-2}{2\theta} \limsup_n \|u_n\|_{\varepsilon_n}^2
	&\leq
	\limsup_n\Big[\Big(\frac{1}{2}-\frac{1}{\theta}\Big)  \|u_n\|_{\varepsilon_n}^2
	+\frac{1}{\theta}\int_{\mathbb{R}^{2}}\Big(f( |u_{n}|^{2}) |u_{n}|^{2}-\frac{\theta}{2}F(| u_{n}|^{2})\Big)dx\Big]\\
	&=
	\lim_n \Big[ J_{\varepsilon_{n}}(u_{n})-\frac{1}{\theta}J_{\varepsilon_{n}}'(u_{n})[u_{n}]\Big]
	=\bar{c}
	\leq c_{V_0}
	<\frac{\theta-2}{16\theta}
\end{align*}
and so, by the diamagnetic inequality \eqref{2.1},
\begin{equation}
	\label{|un|bdd2}
	\limsup_n \|\nabla |u_n| \|_2^2
	\leq
	\limsup_n (\|\nabla |u_n|\|_2^2+V_0\|u_n\|_2^2)
	\leq
	\limsup_n \|u_n\|_{\varepsilon_n}^2
	<\frac{1}{8}.
\end{equation}	
Moreover, since $\bar{c}>0$, we have that $\Vert u_{n}\Vert_{\varepsilon_{n}}\not\rightarrow 0$. Thus, arguing as in Lemma \ref{le35}, we get that there exist a sequence $(\tilde{y}_{n})\subset \mathbb{R}^{2}$ and constants $R$, $\gamma>0$ such that
\begin{equation}\label{4.32}
	\liminf_n\int_{B_{R}(\tilde{y}_{n})}\vert u_{n}\vert^{2}dx\geq \gamma.
\end{equation}
Let $v_{n}:=u_{n} (\cdot+\tilde{y}_{n})$.\\
By \eqref{|un|bdd2} we have that $\{|v_{n}|\}$ is bounded in $H^{1}(\mathbb{R}^{2}, \mathbb{R})$ and so, up to a subsequence, $|v_{n}|\rightharpoonup v$ in $H^{1}(\mathbb{R}^{2}, \mathbb{R})$, with $v\neq 0$ by \eqref{4.32}.\\
Let now $t_{n}>0$ be such that $t_{n}|v_{n}|\in \mathcal{N}_{V_{0}}$ and $y_{n}:=\varepsilon_{n}\tilde{y}_{n}$.\\	
Using the diamagnetic inequality \eqref{2.1}
\[
c_{V_{0}}
\leq I_{V_{0}}(t_{n}|v_{n}|)
= I_{V_0} (t_n |u_n|)
\leq J_{\varepsilon_{n}}(t_n u_{n})
\leq \max_{t\geq 0}J_{\varepsilon_{n}}(tu_{n})
=J_{\varepsilon_{n}}(u_{n})
=c_{\varepsilon_{n}}
\leq c_{V_{0}},
\]
and we conclude.
\end{proof}

\section{Proof of Theorem \ref{gt666} }\label{Sec4}
	In this section, by Ljusternik-Schnirelmann category theory, we prove a multiplicity result for problem \eqref{1.6}.
	
	First, let us provide some useful preliminaries.
	
	Let us start with the following useful compactness result.
	
	\begin{proposition}\label{prop41}
		Let $\varepsilon_{n}\rightarrow 0^+$ and $\{u_{n}\}\subset \mathcal{N}_{\varepsilon_{n}}$ be such that $J_{\varepsilon_{n}}(u_{n})\rightarrow c_{V_{0}}$. Then there exists
		$\{\tilde{y}_{n}\}\subset \mathbb{R}^{2}$ such that the sequence $\{\vert v_n\vert\}\subset H^{1}(\mathbb{R}^{2}, \mathbb{R})$, where $v_{n}(x):= u_{n}(x+\tilde{y}_{n})$, has a convergent subsequence in $H^{1}(\mathbb{R}^{2}, \mathbb{R})$. Moreover, up to a subsequence, $y_{n}:=\varepsilon_{n}\tilde{y}_{n}\rightarrow y\in M$ as $n\to+\infty$.
	\end{proposition}
	\begin{proof}
		Since $\{u_{n}\}\subset \mathcal{N}_{\varepsilon_{n}}$ and $J_{\varepsilon_{n}}(u_{n})\rightarrow c_{V_{0}}$, arguing as in the first part of the proof of Lemma \ref{lemFat} and using the monotonicity of $c_\mu$'s, we have ${\limsup_n}\Vert  u_{n} \Vert_{\varepsilon_n}^2<1/8$ and so, by the diamagnetic inequality  \eqref{2.1},
		\begin{equation}
			\label{|un|bdd}
			\limsup_n \|\nabla |u_n| \|_2^2
			\leq
			\limsup_n (\|\nabla |u_n|\|_2^2+V_0\|u_n\|_2^2)
			<\frac{1}{8}.
		\end{equation}
		Moreover, since $c_{V_{0}}>0$, we have that $\Vert u_{n}\Vert_{\varepsilon_{n}}\not\rightarrow 0$. Thus, arguing as in Lemma \ref{le35}, we get that there exist a sequence $(\tilde{y}_{n})\subset \mathbb{R}^{2}$ and constants $R$, $\beta>0$ such that
		\begin{equation}\label{4.3}
			\liminf_n\int_{B_{R}(\tilde{y}_{n})}\vert u_{n}\vert^{2}dx\geq \beta.
		\end{equation}
		Let $v_{n}(x)=u_{n} (x+\tilde{y}_{n})$. By \eqref{|un|bdd} we have that $\{|v_{n}|\}$ is bounded in $H^{1}(\mathbb{R}^{2}, \mathbb{R})$ and so, up to a subsequence, $|v_{n}|\rightharpoonup v$ in $H^{1}(\mathbb{R}^{2}, \mathbb{R})$, with $v\neq 0$ by \eqref{4.3}.\\
		Let now $t_{n}>0$ be such that $\tilde{v}_{n}:=t_{n}|v_{n}|\in \mathcal{N}_{V_{0}}$ and $y_{n}:=\varepsilon_{n}\tilde{y}_{n}$.\\
		By the diamagnetic inequality \eqref{2.1}
		\[
		c_{V_{0}}
		\leq I_{V_{0}}(\tilde{v}_{n})
		= I_{V_0} (t_n |u_n|)
		\leq J_{\varepsilon_{n}}(t_n u_{n})
		\leq \max_{t\geq 0}J_{\varepsilon_{n}}(tu_{n})=J_{\varepsilon_{n}}(u_{n})=c_{V_{0}}+o_{n}(1),
		\]
		and so $I_{V_{0}}(\tilde{v}_{n})\rightarrow c_{V_{0}}$ as $n\to +\infty$.\\
		Thus, by Lemma \ref{lemFat} for $\mu=V_{0}$, $\{\tilde{v}_{n}\}$ is bounded in $H^{1}(\mathbb{R}^{2}, \mathbb{R})$ and, since $\vert v_{n}\vert\not\rightarrow 0$ in $H^{1}(\mathbb{R}^{2}, \mathbb{R})$, then $\{t_{n}\}$ is bounded too and so, up to a subsequence, $t_{n}\rightarrow t_{0}\geq 0$.\\
		Actually we have that $t_{0}>0$.\\
		Indeed, if $t_{0}=0$, then, due to the boundedness of $\{\vert v_{n}\vert\}$ in $H^{1}(\mathbb{R}^{2}, \mathbb{R})$, $\tilde{v}_{n}\rightarrow 0$ in $H^{1}(\mathbb{R}^{2}, \mathbb{R})$, so that $I_{V_{0}}(\tilde{v}_{n})\rightarrow 0$, in contradiction with $c_{V_{0}}>0$.\\
		Hence, up to a subsequence, $\tilde{v}_{n}\rightharpoonup \tilde{v}:=t_{0} v\neq 0$ in $H^{1}(\mathbb{R}^{2}, \mathbb{R})$ and so, by Lemma \ref{lemFat}, $\tilde{v}_{n}\rightarrow \tilde{v}$ in $H^{1}(\mathbb{R}^{2}, \mathbb{R})$, which gives $|v_{n}|\rightarrow v$ in $H^{1}(\mathbb{R}^{2},
		\mathbb{R})$.\\
		Let us show now that $\{y_{n}\}$ is bounded.\\
		Assume by contradiction that $\{y_{n}\}$ is not bounded.  Then there exists a subsequence such that $\vert y_{n}\vert\rightarrow+\infty$.\\
		If $V_{\infty}=+\infty$,
		since, by (\ref{V}),
		\[
		\inf_{B_R(\tilde{y}_n)} V_{\varepsilon_{n}}
		= \inf_{B_{R\varepsilon_{n}}(y_n)} V
		\to +\infty,
		\]
		then
		\begin{align*}
			\int_{\mathbb{R}^{2}} V(\varepsilon_{n}x+y_{n})\vert v_{n}\vert^{2}dx
			&\geq
			\int_{B_R(0)} V(\varepsilon_{n}x+y_{n})\vert v_{n}\vert^{2}dx
			=
			\int_{B_R(0)} V(\varepsilon_{n}x+y_{n})\vert u_{n} (x+\tilde{y}_n)\vert^{2}dx\\
			&=
			\int_{B_R(\tilde{y}_n)} V(\varepsilon_{n}z)\vert u_{n} (z)\vert^{2}dz
			\geq
			\inf_{B_R(\tilde{y}_n)} V_{\varepsilon_{n}}
			\int_{B_R(\tilde{y}_n)} \vert u_{n} (z)\vert^{2}dz
			\to +\infty,
		\end{align*}
		in contradiction with
		\[
		\int_{\mathbb{R}^{2}}V(\varepsilon_{n}x+y_{n})\vert v_{n}\vert^{2}dx
		\leq \Vert u_{n}\Vert_{\varepsilon_{n}}^{2}\leq C.
		\]
		If $V_{\infty}\in\mathbb{R}$, since  $\tilde{v}_{n}\rightarrow \tilde{v}$ in $H^{1}(\mathbb{R}^{2} , \mathbb{R})$, by the Fatou's Lemma and since $H^{1}(\mathbb{R}^{2} , \mathbb{R})$ is continuously embedded in $L^2(\mathbb{R}^{2} , \mathbb{R})$,
		\[
		V_\infty \lim_n \int_{\mathbb{R}^{2}} \vert \tilde{v}_n\vert^{2}dx
		= V_\infty \int_{\mathbb{R}^{2}} \vert \tilde{v}\vert^{2}dx
		\leq
		\liminf_n\int_{\mathbb{R}^{2}} V(\varepsilon_{n}x+y_{n})\vert \tilde{v}_{n}\vert^{2}dx,
		\]
		and so, by the diamagnetic inequality \eqref{2.1} and \eqref{Nehepsmax},
		\begin{equation}\label{inequality}
			\begin{split}
				c_{V_{0}}
				&=I_{V_{0}}(\tilde{v})< I_{V_{\infty}}(\tilde{v})
				= \lim_n I_{V_{\infty}}(\tilde{v}_n)\\
				&
				\leq
				\liminf_n \frac{1}{2}\int_{\mathbb{R}^{2}} V(\varepsilon_{n}x+y_{n})\vert \tilde{v}_{n}\vert^{2}dx
				- \frac{V_\infty}{2} \int_{\mathbb{R}^{2}} \vert \tilde{v}\vert^{2}dx
				+ \lim_n I_{V_{\infty}}(\tilde{v}_n)\\
				&\leq
				{\liminf_n}\Big(\frac{1}{2}\int_{\mathbb{R}^{2}}(\vert\nabla   \tilde{v}_{n}\vert^{2}
				+ V(\varepsilon_{n}x+y_{n})\vert 		\tilde{v}_{n}\vert^{2})dx
				-\frac{1}{2}\int_{\mathbb{R}^{2}}F(\vert \tilde{v}_{n}\vert^{2})dx\Big)\\
				&=
				{\liminf_n}\Big(\frac{t^{2}_{n}}{2}\int_{\mathbb{R}^{2}}(\vert\nabla  \vert u_{n}\vert\vert^{2}
				+ V_{\varepsilon_{n}}(z)\vert 		u_{n}\vert^{2})dz
				-\frac{1}{2}\int_{\mathbb{R}^{2}}F(t^{2}_{n}\vert u_{n}\vert^{2})dz\Big)\\
				&\leq {\liminf_n}\,J_{\varepsilon_{n}}(t_{n}u_{n})\leq {\liminf_n}\,J_{\varepsilon_{n}}(u_{n})=c_{V_{0}},
			\end{split}
		\end{equation}
		which gives a contradiction.\\
		Hence, up to a subsequence, $y_{n}\rightarrow y\in \mathbb{R}^{2}$.\\
		If $y\not\in M$, then
		$V(y)>V_{0}$ and we can argue as in \eqref{inequality} to get a contradiction and so the proof is complete.
	\end{proof}

	Let now $\delta>0$ be fixed, $\omega\in H^{1}(\mathbb{R}^{2}, \mathbb{R})$ be a positive ground state solution of problem \eqref{3.2} for $\mu=V_{0}$, and $\eta\in C^{\infty}([0, +\infty), [0, 1])$ be a non-increasing cut-off function such that $\eta(t)=1$ if $0\leq t\leq \delta/{2}$ 	and $\eta(t)=0$ if $t\geq \delta$.
	
	For any $y\in M$, let
	\begin{equation*}
		\Psi_{\varepsilon, y}(x):=\eta(\vert \varepsilon x-y\vert)\omega\Big(\frac{\varepsilon x-y}{\varepsilon}\Big)\exp\Big(i\tau_{y}\Big(\frac{\varepsilon x-y}{\varepsilon}\Big)\Big),
	\end{equation*}
	where $M$ is defined in \eqref{M} and $\tau_{y}(x):=A_{1}(y)x_{1}+A_{2}(y)x_{2}$.
	
	Let $t_{\varepsilon,y}>0$ be the unique positive number such that $t_{\varepsilon,y}\Psi_{\varepsilon, y}\in \mathcal{N}_{\varepsilon}$ (and so
	$\max_{t\geq 0}J_{\varepsilon}(t\Psi_{\varepsilon, y})=J_{\varepsilon,y}(t_{\varepsilon, y}\Psi_{\varepsilon, y})$) and  let us consider
	\[
	\Phi_{\varepsilon}:=y\in M\mapsto t_{\varepsilon,y}\Psi_{\varepsilon, y}\in \mathcal{N}_{\varepsilon}.
	\]
	Observe that,  by the assumptions on $V$, $M$ is compact and, for any $y\in M$, $\Phi_{\varepsilon}(y)$ has compact support.\\
	Moreover we have
	\begin{lemma}\label{lem51}
		As $\varepsilon\rightarrow 0^{+}$,
		$$
		\Vert \Psi_{\varepsilon,y}\Vert_{\varepsilon}^{2}\rightarrow \int_{\mathbb{R}^{2}}(\vert \nabla \omega\vert^{2}+V_{0} \omega^{2})dx
		$$
		uniformly with respect to  $y\in M$.
	\end{lemma}
	\begin{proof}
		First we observe that
		\[
		\sup_{y\in M} \Big|
		\int_{\mathbb{R}^{2}}V_{\varepsilon}( x)\vert \Psi_{\varepsilon,y}\vert^{2}dx
		- V_{0} \int_{\mathbb{R}^{2}}\omega^{2}dx
		\Big| \to 0
		\text{ as } \varepsilon\rightarrow 0^{+}.
		\]
		Indeed, taking $z=(\varepsilon x-y)/\varepsilon$,
		\[
		\sup_{y\in M} \Big|
		\int_{\mathbb{R}^{2}}\vert \Psi_{\varepsilon,y}\vert^{2}dx
		-  \int_{\mathbb{R}^{2}}\omega^{2}dx
		\Big|
		\leq
		\int_{\mathbb{R}^{2}} |\eta^2(\vert\varepsilon z \vert )-1|\omega^2(z)dz
		\leq
		\int_{B_{\delta/(2\varepsilon)}^c(0)} \omega^2(z)dz\to 0.
		\]
		Moreover
		\begin{align*}
			\sup_{y\in M} \Big|
			\int_{\mathbb{R}^{2}}[V_{\varepsilon}( x)-V_0]\vert \Psi_{\varepsilon,y}\vert^{2}dx
			\Big|
			&\leq
			\sup_{y\in M}
			\int_{\mathbb{R}^{2}}
			|V(\varepsilon z+y)  -V_0|\eta^2(\vert\varepsilon z \vert )\omega^2(z)dz\\
			&\leq
			\int_{B_{\delta/\varepsilon}(0)}\sup_{y\in M}
			|V(\varepsilon z+y)  -V_0|\omega^2(z)dz\\
			&=
			\int_{\mathbb{R}^2} \sup_{y\in M} |V(\varepsilon z+y)  -V_0|\chi_{B_{\delta/\varepsilon}(0)}(z)\omega^2(z)dz
			\to 0,
		\end{align*}
		applying the Lebesgue Dominated Convergence Theorem, since, if $|\varepsilon z|<\delta$, then, due to the compactness of $M$, $|\varepsilon z+y|$ is bounded when $y\in M$,
		\[
		\sup_{y\in M} |V(\varepsilon z+y)  -V_0|\chi_{B_{\delta/\varepsilon}(0)}(z)\omega^2(z)
		\leq
		C \omega^2(z)\in L^1(\mathbb{R}^2,\mathbb{R}),
		\]
		and $\sup_{y\in M} |V(\varepsilon z+y)  -V_0|\chi_{B_{\delta/\varepsilon}(0)}(z)\omega^2(z)$ converges to $0$ a.e. in $\mathbb{R}^2$ as $\varepsilon\rightarrow 0^{+}$.\\
		Then, since, taking again $z=(\varepsilon x-y)/\varepsilon$,
		\begin{align*}
			\int_{\mathbb{R}^{2}}\vert \nabla_{A_{\varepsilon}}\Psi_{\varepsilon,y}\vert^{2}dx
			&=
			\varepsilon^2 \int_{\mathbb{R}^{2}}| \eta'(\vert \varepsilon z\vert)|^2\omega^2(z) dz
			+\int_{\mathbb{R}^{2}}  \eta^2(\vert \varepsilon z\vert) |\nabla\omega(z)|^{2}dz\\
			&\qquad
			+\int_{\mathbb{R}^{2}} |A(y)-A(\varepsilon z+y)|^2  \eta^2(\vert \varepsilon z\vert)\omega^2(z)dz\\
			&\qquad
			+2\varepsilon \int_{\mathbb{R}^{2}}
			\eta(\vert \varepsilon z\vert) \eta'(\vert \varepsilon z\vert)\omega(z)\nabla\omega(z)\cdot \frac{z}{|z|}dz,
		\end{align*}
		and,
		\[
		\Big|
		\int_{\mathbb{R}^{2}} \eta^{2}(\vert \varepsilon z\vert)|\nabla\omega(z)|^{2}dz
		-\int_{\mathbb{R}^{2}}\vert\nabla\omega(z)\vert^{2}dz
		\Big|
		\leq
		\int_{\mathbb{R}^{2}} |\eta^2(\vert\varepsilon z \vert )-1|\vert\nabla\omega(z)\vert^{2}dz
		\leq
		\int_{B_{\delta/(2\varepsilon)}^c(0)} \vert\nabla\omega(z)\vert^{2}dz,
		\]
		\[
		\int_{\mathbb{R}^{2}}|\eta'(\vert \varepsilon z\vert)|^2 \omega^2(z)dz
		\leq
		C\int_{B_{\delta/(2\varepsilon)}^c(0)} \omega^{2}(z)dz,
		\]
		\[
		\int_{\mathbb{R}^{2}}|\eta(\vert \varepsilon z\vert)\eta'(\vert \varepsilon z\vert)\omega(z)\nabla\omega(z)|dz
		\leq
		C\int_{B_{\delta/(2\varepsilon)}^c(0)} \omega(z)|\nabla\omega|dz,
		\]
		
		\[
		\sup_{y\in M} \int_{\mathbb{R}^{2}} |A(y)-A(\varepsilon z+y)|^2  \eta^2(\vert \varepsilon z\vert)\omega^2(z)dz
		\leq
		\int_{B_{\delta/\varepsilon}(0)} \sup_{y\in M} |A(y)-A(\varepsilon z+y)|^2\omega^{2}(z)dz\to 0,
		\]
		due to the continuity of $A$ and a similar argument as before.\\
		Thus we infer
		\[
		\sup_{y\in M} \Big|
		\int_{\mathbb{R}^{2}}\vert \nabla_{A_{\varepsilon}}\Psi_{\varepsilon,y}\vert^{2}dx
		- \int_{\mathbb{R}^{2}}\vert \nabla \omega\vert^{2}dx
		\Big|
		\to 0
		\]
		as $\varepsilon\rightarrow 0^{+}$ and we conclude.
	\end{proof}

Hence we can conclude in the following way.

	\begin{proof} [Proof of Theorem \ref{gt666}]
Arguing as in \cite[Lemma~4.1]{rDJ}, and using Lemma \ref{lem51}, we can prove that $J_{\varepsilon}(\Phi_{\varepsilon}(y))\to c_{V_{0}}$ uniformly with respect to  $y\in M$ as $\varepsilon\rightarrow 0^{+}$.\\
	Then, given $\delta>0$, if $\rho>0$ with $M_{\delta}\subset B_{\rho}$, let
	\[
	\Upsilon: \mathbb{R}^{2}\rightarrow \mathbb{R}^{2},
	\qquad
	\Upsilon(x):=
	\begin{cases}
		x, & \text{if } \vert x\vert<\rho,\\
		\rho x/\vert x\vert , & \text{if } \vert x\vert\geq\rho
	\end{cases}
	\]
	and then let us define the barycenter map as
	\[
	\beta_{\varepsilon}:= u\in\mathcal{N}_{\varepsilon}\mapsto \frac{1}{\Vert u\Vert_2^{2}}\int_{\mathbb{R}^{2}}\Upsilon(\varepsilon x)\vert u(x)\vert^{2} dx \in \mathbb{R}^{2}.
	\]
	By \cite[Lemma~4.2]{rDJ} we have that $\beta_{\varepsilon}(\Phi_{\varepsilon}(y))\to y$ as $\varepsilon\rightarrow 0^{+}$ uniformly in $ y\in M$.\\
	Let now $h: \mathbb{R}^{+}\rightarrow \mathbb{R}^{+}$ be any positive function such that $h(\varepsilon)\rightarrow 0$ as $\varepsilon\rightarrow 0^{+}$ and
	\begin{equation*}
		\tilde{\mathcal{N}}_{\varepsilon}:=\{u\in \mathcal{N}_{\varepsilon}: J_{\varepsilon}(u)\leq c_{V_{0}}+h(\varepsilon)\}.
	\end{equation*}
	Since $\vert J_{\varepsilon}(\Phi_{\varepsilon}(y))-c_{V_{0}}\vert\rightarrow 0$ uniformly in $ y\in M$ as $\varepsilon\rightarrow 0^{+}$, then $\tilde{\mathcal{N}}_{\varepsilon}\neq\emptyset$ for $\varepsilon>0$ small enough.\\
	Thus, arguing as in \cite[Lemma~4.4]{rDJ},
	\begin{equation*}
		\lim_{\varepsilon\rightarrow 0^{+}}\sup_{u\in \tilde{\mathcal{N}}_{\varepsilon}}\operatorname{dist}(\beta_{\varepsilon}(u), M_{\delta})=0.
	\end{equation*}
	Hence, as in \cite[Section~6]{rCL}, we can find $\tilde{\varepsilon}_{\delta}>0$ such that for any $\varepsilon\in (0, \tilde{\varepsilon}_{\delta})$, the following diagram
	\[
	M \xrightarrow{\Phi_{\varepsilon}}\tilde{\mathcal{N}}_{\varepsilon}\xrightarrow{\beta_{\varepsilon}} M_{\delta}
	\]
	is well defined and $\beta_{\varepsilon}\circ \Phi_{\varepsilon}$ is homotopically equivalent to the embedding $\iota: M\rightarrow M_{\delta}$. Thus, \cite[Lemma~4.3]{rBC} (see also \cite[Lemma 2.2]{rCLJDE}) implies that
	\begin{equation}\label{category}
		\text{cat}_{\tilde{\mathcal{N}}_{\varepsilon}}(\tilde{\mathcal{N}}_{\varepsilon})\geq \text{cat}_{M_{\delta}}(M).
	\end{equation}
Now observe that $J_\varepsilon$ is bounded from below on  $\mathcal{N}_\varepsilon$  (see the preliminary results in Section \ref{Sec3} and in particular \eqref{minmaxNehari}). Moreover, using, respectively, Lemma \ref{newadd}, Lemma \ref{le32}, and \eqref{moncmu}, for $\varepsilon$ small enough,
\[
c_\varepsilon \leq c_{V_0} <c_{V_0} + h(\varepsilon),
\]
\[
c_{V^*} + h(\varepsilon) <\frac{\theta-2}{16\theta},
\]
and, if $V_\infty <+\infty$,
\begin{equation}
\label{cVinfty}
c_{V_0} + h(\varepsilon)< c_{V_\infty}.
\end{equation}
Thus, if $\{u_{n}\}$ is a $(PS)_{c}$ sequence for $J_\varepsilon|_{\mathcal{N}_\varepsilon}$ with $c\in[c_\varepsilon,c_{V_0}+h(\varepsilon)]$, then, by (\ref{f3}) and \eqref{moncmu},
\begin{align*}
\frac{\theta-2}{2\theta}\|u_n\|_{\varepsilon}^2
& \leq
\Big(\frac{1}{2}-\frac{1}{\theta}\Big) \|u_n\|_{\varepsilon}^2
+\frac{1}{\theta}\int_{\mathbb{R}^{2}}\Big(f( |u_{n}|^{2}) |u_{n}|^{2}-\frac{\theta}{2}F( |u_{n}|^{2})\Big)dx\\
& =
J_{\varepsilon}(u_{n})-\frac{1}{\theta}J_{\varepsilon}'(u_{n})[u_{n}]
= c+o_n(1)\\
&<
c_{V_0} + h(\varepsilon) + o_n(1)
\leq c_{V^*} + h(\varepsilon) + o_n(1).
\end{align*}
Then, passing to the $\limsup$ on $n$ we have
\[
\frac{\theta-2}{2\theta} \limsup_{n} \|u_n\|_{\varepsilon}^2
\leq c_{V^*} + h(\varepsilon)
<\frac{\theta-2}{16\theta}
\]
for $\varepsilon$ small enough and so
\[
\limsup_{n}\Vert u_{n}\Vert_{\varepsilon}^2< \frac{1}{8}.
\]
Hence, using \eqref{cVinfty}  if $V_\infty <+\infty$, we can apply Proposition \ref{gt621}, obtaining that $J_\varepsilon|_{\mathcal{N}_\varepsilon}$ satisfies the $(PS)_{c}$ condition for $c\in[c_\varepsilon,c_{V_0}+h(\varepsilon)]$. Then, by \cite[Theorem~5.20]{rW}, the Ljusternik-Schnirelmann theory for $C^{1}$ functionals implies that $\tilde{\mathcal{N}}_{\varepsilon}$ contains at least $\text{cat}_{\tilde{\mathcal{N}}_{\varepsilon}}(\tilde{\mathcal{N}}_{\varepsilon})$ critical points of $J_\varepsilon|_{\mathcal{N}_\varepsilon}$. Observing that, due to Corollary \ref{gt678}, such critical points are critical points for $J_{\varepsilon}$ and using \eqref{category}, we conclude the proof of the multiplicity result.
\end{proof}

\section{Proof of Theorem \ref{thm3}}

This last section is dedicated to the study of the behavior of the maximum points  of the modulus of the solutions obtained in Theorem \ref{gt62} and in Theorem \ref{gt666} and to their the decay property.

\begin{proof}[Proof of Theorem \ref{thm3}]
Let $\varepsilon_{n}\rightarrow 0^{+}$ as $n\rightarrow +\infty$ and let $u_{\varepsilon_{n}}$ be a nontrivial solution of \eqref{1.6}  found in Theorem \ref{gt62} or in Theorem \ref{gt666}. For ground states found in Theorem \ref{gt62}, by Proposition \ref{propcecv0}, up to a subsequence, we have that $J_{\varepsilon_{n}}(u_{\varepsilon_{n}})=c_{\varepsilon_{n}}\to c_{V_0}$ and, for solutions found in Theorem \ref{gt666} we have that $J_{\varepsilon_{n}}(u_{\varepsilon_{n}})\leq c_{V_{0}}+h(\varepsilon_{n})$.
Thus, Proposition \ref{prop41} applies and so we can consider $\{\tilde{y}_{n}\}\subset \mathbb{R}^{2}$, and
 $v_{\varepsilon_{n}}:=u_{\varepsilon_{n}}(\cdot+\tilde{y}_{n})$, which  is a solution of
\[
\Big(\frac{1}{i}\nabla-A_{\varepsilon_n}(x+\tilde{y}_{n})\Big)^{2}v_{\varepsilon_{n}}
+V_{\varepsilon_n}(x+\tilde{y}_{n})v_{\varepsilon_{n}}
=f(\vert v_{\varepsilon_{n}}\vert^{2})v_{\varepsilon_{n}}
\quad
\hbox{in }\mathbb{R}^2,
\]
such that, up to a subsequence, $|v_{\varepsilon_{n}}|\rightarrow v$ in $H^{1}(\mathbb{R}^{2}, \mathbb{R})$ and $y_{n}=\varepsilon_{n}\tilde{y}_{n} \rightarrow y \in M$.\\
Arguing as in \cite[Lemma 5.1]{rDJ},  we have that $\{\vert v_{\varepsilon_{n}}\vert\}$ is bounded in $L^{\infty}(\mathbb{R}^{2}, \mathbb{R})$ and
\begin{equation}
\label{vnvan}
\lim_{\vert x\vert\rightarrow+\infty} \vert v_{\varepsilon_{n}}(x)\vert=0\quad\text{uniformly in } n\in \mathbb{N}.
\end{equation}
In addition, by the first part of the proof of Proposition \ref{prop41}, there exist $R,\beta>0$ such that
\[
{\liminf_n}\int_{B_{R}(\tilde{y}_{n})}\vert u_{\varepsilon_{n}}\vert^{2}dx\geq \beta.
\]
Then, if $\kappa>0$ such that $0<\pi \kappa^2 R^2  <\beta$, for $n$ large enough,
\[
\pi \kappa^2 R^2
\leq\int_{B_{R}(\tilde{y}_{n})}\vert u_{\varepsilon_{n}}\vert^{2}dx
\leq \pi R^2 \Vert  v_{\varepsilon_{n}}\Vert^{2}_{\infty},
\]
and so
\begin{equation}
\label{vnkappa}
\Vert  v_{\varepsilon_{n}} \Vert_{\infty}\geq \kappa.
\end{equation}
Let us show the concentration  of global maximum points  $p_n$ of $\vert v_{\varepsilon_{n}}\vert$.\\
If $z_n:=p_{n}+\tilde{y}_{n}$, by the definition of $v_{\varepsilon_n}$, we have that each $z_n$ is a global maximum point of $\vert u_{\varepsilon_{n}}\vert$.
Moreover, since by \eqref{vnvan} and \eqref{vnkappa}, $p_{n}\in B_{R}(0)$ for some $R>0$,  $\eta_{n}:=\varepsilon_{n}z_{n}=\varepsilon_{n}p_{n}+\varepsilon_{n}\tilde{y}_{n}\rightarrow y\in M$ and so
\begin{equation*}
	\lim_{n}V(\eta_{n})=V_{0}.
\end{equation*}
Now let us assume that $A\in C^{1}(\mathbb{R}^{2}, \mathbb{R}^{2})$ and let us  show decay of the modulus of our solutions.\\
Arguing, for instance, as in \cite[Remark 1.2]{rCJS}, $v_{\varepsilon_{n}}\in C^{1}(\mathbb{R}^{2}, \mathbb{C})$. Moreover, by \eqref{vnvan} and (\ref{f1}), there exists a $R_{1}>0$ such that
\begin{equation}\label{re1}
f(\vert v_{\varepsilon_{n}}\vert^{2})
\leq \frac{V_{0}}{2}
\text{ for all } \vert x\vert\geq R_{1} \text{ and uniformly in } n\in \mathbb{N}.
\end{equation}
Since $v_{\varepsilon_{n}}\in C^{1}(\mathbb{R}^{2}, \mathbb{C})$ and \eqref{vnvan} holds, there exist $c,C>0$ with $c^{2}<V_{0}/2$ such that
\[
\vert v_{\varepsilon_{n}}(x)\vert \leq Ce^{-c R_{1}}
\text{ for } \vert x\vert= R_{1}.
\]
Let $\varphi(x):=Ce^{-c\vert x\vert}$ in $\mathbb{R}^2$.\\
Observe that
\begin{equation}\label{re2}
\Delta \varphi\leq c^{2}\varphi
\text{ in } \mathbb{R}^2\setminus\{0\}.
\end{equation}
Moreover, by the Kato's inequality (see \cite[Lemma A]{rKa}) and  \eqref{V}, $\vert v_{\varepsilon_{n}}\vert$ satisfies (in the weak sense)
\[
-\Delta \vert v_{\varepsilon_{n}}\vert
+V_{0}\vert v_{\varepsilon_{n}}\vert
\leq f(\vert v_{\varepsilon_{n}}\vert^{2})\vert v_{\varepsilon_{n}}\vert
\hbox{ in }\mathbb{R}^2.
\]
Thus, by \eqref{re1}, we have that, uniformly in $n\in \mathbb{N}$,
\begin{equation}\label{re3}
-\Delta \vert v_{\varepsilon_{n}}\vert
+ \frac{V_{0}}{2}\vert v_{\varepsilon_{n}}\vert
\leq 0
\text{ for } \vert x\vert\geq R_{1}.
\end{equation}
Let, now, $\varphi_{n}:=\varphi-\vert v_{\varepsilon_{n}}\vert$.\\
By \eqref{re2} and \eqref{re3}, we have
\[
\begin{cases}
\displaystyle -\Delta \varphi_{n}+\frac{V_{0}}{2} \varphi_{n}\geq 0
&\text{ for } \vert x\vert\geq R_{1},\\
\varphi_{n}\geq 0
&\text{ for }\vert  x\vert= R_{1},\\
\displaystyle \lim_{\vert x\vert\rightarrow+\infty}\varphi_{n}(x)=0.
\end{cases}
\]
The classical maximum  principle (see e.g. \cite[Theorem 7.4.2]{rC1}) implies that $\varphi_{n}\geq 0$ for $\vert x\vert\geq R_{1}$ and so we obtain that
\[
|u_{\varepsilon_{n}}(x)|
=
\vert v_{\varepsilon_{n}}(x-\tilde{y}_n)\vert
\leq
Ce^{-c\vert x -\tilde{y}_n\vert}
\text{ for all }\vert x\vert\geq R_{1}
\text{ and uniformly in } n\in \mathbb{N}.
\]
Observe that, due to the change of variable introduced in Section \ref{Sec2} to write \eqref{1.1} in the equivalent form \eqref{1.6}, the points $\eta_n$ are maximizers of solutions of \eqref{1.1} and \eqref{decay} holds.
\end{proof}

\section*{Acknowledgments}
P. d'Avenia is supported by PRIN project 2017JPCAPN {\em Qualitative and quantitative aspects of nonlinear PDEs}. C. Ji is partially supported by Shanghai Natural Science Foundation (20ZR1413900).

\end{document}